\documentclass[12pt]{amsart}

\usepackage{amsmath}
\usepackage{amsthm}
\usepackage{amssymb}
\usepackage{amsfonts}
\usepackage{bm}
\usepackage[shortlabels]{enumitem}
\usepackage[bookmarks=true,hyperindex,pdftex,colorlinks,citecolor=red, linkcolor=blue]{hyperref}
\usepackage{centernot} 
\usepackage{marginnote} 
\usepackage[left=3cm,right=3cm,top=3cm,bottom=2.5cm]{geometry}
\usepackage{bbm}
\usepackage[all]{xy}
\usepackage{tikz}
\usepackage{comment}
\usetikzlibrary{arrows}
\usetikzlibrary{shapes,decorations}
\usetikzlibrary{positioning}
\usepackage{xcolor}
\usepackage{bigints}
\usepackage{enumitem}
\usepackage{makecell}
\usepackage{arydshln}

\theoremstyle{plain}
\newtheorem{thm}{Theorem}[section]
\newtheorem{prop}[thm]{Proposition}
\newtheorem{lemma}[thm]{Lemma}
\newtheorem{cor}[thm]{Corollary}
\newtheorem{remark}[thm]{Remark}

\usepackage{makecell}
\theoremstyle{definition}
\newtheorem{df}[thm]{Definition}
\newtheorem{q}[thm]{Question}

\newcommand{\N}{\mathbb{N}}

\newcommand{\D}{\mathbb{D}}

\newcommand{\norm}[1]{\left\Vert #1\right\Vert}
\usepackage{multirow}
\definecolor{darkgreen}{rgb}{.2, .6, .2}

\begin{document}
	
	\title{Rates of Decay for $(\alpha, \beta)$-Ritt-Kreiss Operators}
	
	\author[L. Arnold]{Loris Arnold}

	\address[L. Arnold]{Normandie Univ, UNICAEN, CNRS, LMNO, 14000 Caen, France}
	\email{lfj.arld@gmail.com}

	\date{} 
	
	\begin{abstract}
		This paper investigates the growth of the sequences $(\|T^n(I-T)^k\|)_{n\ge 1}$ for bounded linear operators on Banach spaces under various resolvent conditions. Focusing on $\alpha$-Ritt operators that are also strongly Kreiss bounded, we show that the growth estimates established by Nevanlinna for Kreiss bounded operators can be significantly refined within the Hilbert space setting, nearly attaining the optimal rates obtained by Seifert for the more restrictive power-bounded case. We further introduce the class of $(\alpha, \beta)$-RK operators as a generalization of both Ritt and Kreiss-type conditions. For these operators, we derive comprehensive growth estimates which, for certain ranges of $\alpha$ and $\beta$, yield improvements over existing bounds in the literature.
		
		Particular attention is given to $\beta$-Kreiss operators, for which we provide a characterization via Cesàro-type means. We show that, in contrast to the well-known case $\beta = 1$, the power growth estimate $\|T^n\| = O(n^{\beta})$ is sharp whenever $\beta > 1$. The optimality of our estimates is discussed in several cases, relying on constructions and techniques developed by Nevanlinna, Spijker, and Borovykh. We conclude by providing a characterization of Ritt operators that appears to be absent from the literature.
	\end{abstract}

	\subjclass[2020]{Primary 47A05, 47A10; Secondary 47D03, 47A35, 34D05, 35B40}
	

	\keywords{Ritt condition, Kreiss condition, Cesàro
		boundedness, rate of decay, }
	
	\maketitle

	\section{Introduction}
	
	The classical Katznelson-Tzafriri theorem asserts that for a power-bounded operator $T$ on a Banach space $X$, the sequence $\|T^n(I-T)\|$ converges to zero as $n \to \infty$ if and only if the spectrum of $T$ satisfies $\sigma(T)\cap \mathbb{T} \subset \{1\}$. To obtain quantitative rates for this decay, additional conditions on the resolvent are required. A prominent class in this context is that of $\alpha$-Ritt operators for $\alpha \ge 1$, defined by the condition:
	\begin{equation}\label{alpharitt}
		\sigma(T) \subset \overline{\D} \quad \text{and} \quad \|R(\lambda,T)\| \le \frac{C}{|\lambda-1|^\alpha}, \quad 1<|\lambda|<2.
	\end{equation}
	Under the assumption of power-boundedness, Seifert \cite{Sei} proved that 
	\begin{equation}\label{seifert}
		\|T^n(I-T)\| = O(n^{-1/\alpha} \log(n+1)).
	\end{equation}
	The case $\alpha=1$ corresponds to the classical Ritt operators, which can be characterized as those power-bounded operators for which the optimal decay $\|T^n(I-T)\| = O(n^{-1})$ holds. More generally, Ritt operators admit the following equivalent characterizations:
	
	\begin{thm}\label{Ritt}
		Let $T\in B(X)$ with $\sigma(T) \subset \overline{\D}$. The following are equivalent:
		\begin{enumerate}[label=(\roman*)]
			\item $T$ is a Ritt operator;
			\item For $\alpha,\beta>0$ with $\alpha+\beta = 1$, 
			$$\|R(\lambda,T)\| \le \frac{C}{|\lambda-1|^\alpha (|\lambda|-1)^\beta}, \quad  1<|\lambda|<2;$$ 
			\item $T$ is power-bounded and, for some (or all) $k\in \mathbb{N}$, $\|T^n(I-T)^k\| = O(n^{-k})$.
		\end{enumerate}
	\end{thm}
	
	A natural question is whether similar estimates for $\alpha$-Ritt operators hold when the assumption of power-boundedness is relaxed. An operator $T \in B(X)$ is said to be Kreiss bounded if
	\begin{equation}\label{Kreiss}
		\sigma(T)\subset \overline{\mathbb{D}} \quad \text{and} \quad \|R(\lambda,T)\|\le \frac{C}{|\lambda|-1}, \quad |\lambda|>1.
	\end{equation}
	While Kreiss boundedness is equivalent to power-boundedness in finite dimensions, it is strictly weaker in infinite-dimensional spaces. In the latter case, the growth of $\|T^n\|$ can be as fast as $O(n)$, although more refined bounds exist depending on the geometry of the space $X$. For example, if $X$ is a UMD space with finite cotypes $q$ and $q^*$ for $X$ and $X^*$ respectively, and we set $s = \min(q, q^*)$, Theorem 3.1 in \cite{Cuny} states that
	\begin{equation*}\label{KreissBound}
		\|T^n\| = O\Bigg(\frac{n}{\log^{1/s}(n+1)}\Bigg).
	\end{equation*}
	A stronger property, known as strong Kreiss boundedness, was introduced by McCarthy \cite{McCarthy}. It is defined by the following estimate on the powers of the resolvent:
	\begin{equation*}\label{StronglyKreiss}
		\sigma(T)\subset \overline{\mathbb{D}} \quad \text{and} \quad \|R(\lambda,T)^n\|\le \frac{C}{(|\lambda|-1)^n}, \quad |\lambda|>1, \ n\in \mathbb{N}.
	\end{equation*}
	This condition has been the subject of extensive study, and various growth estimates for $\|T^n\|$ have been obtained depending on the geometry of the underlying space $X$ (see  \cite{Nev, CCEL, ArnCun, DLV}):
	\begin{equation*}\label{skresti}
		\|T^n\| = 
		\begin{cases}
			O(n^{1/2}), & \text{in general},\\
			O(\log^{\kappa}(n+1)), & X \text{ a Hilbert space},\\
			O\big(n^{|1/2-1/p|}\log^{\kappa}(n+1)\big), & X \text{ an } L^p\text{-space},\\
			o(n^{1/2}), & X \text{ a UMD space}.
		\end{cases}
	\end{equation*}
	
	One of the primary goals of this work is to study $\alpha$-Ritt operators that satisfy these weaker resolvent conditions. For Kreiss bounded $\alpha$-Ritt operators, Nevanlinna \cite[Theorem 9.]{Nev} established the following growth estimate:
	\begin{thm}\label{Nevanlinna}
		Let $T$ be Kreiss bounded and satisfying \eqref{alpharitt}. Then, for $k\in \mathbb{N}\cup \{0\}$,
		\begin{equation}\label{NevEsti}
			\|T^n(I-T)^k\| = O\big(n^{\frac{\alpha-k-1}{\alpha}}\big).
		\end{equation}
	\end{thm}

	The paper is organized as follows.
	
	In Section 2, we show that these estimates \eqref{NevEsti} can be significantly refined by assuming strong Kreiss boundedness and employing moment inequalities. In the Hilbert space setting, we derive a bound analogous to \eqref{seifert} for strongly Kreiss bounded $\alpha$-Ritt operators:
	$$
	\|T^n(I-T)^k\| = O\left( \frac{\log^{\kappa}(n+1)}{n^{k/\alpha}} \right), \quad k \in \N\cup\{0 \}.
	$$
	
	Section 3 is devoted to a complete proof of Theorem \ref{Nevanlinna}. This result serves as a foundational tool for the derivations presented in the subsequent sections.

	In Section 4, we introduce the class of $(\alpha,\beta)$-RK operators, defined by the resolvent estimate
	\begin{equation*}
		\|R(\lambda,T)\| \le \frac{C}{|\lambda-1|^\alpha (|\lambda|-1)^\beta}, \quad 1<|\lambda|<2,
	\end{equation*}
	which unifies Ritt-type ($\beta=0$) and Kreiss-type ($\alpha=0$) conditions. For the parameter range $\alpha>0$, $0 \le \beta < 1$ with $\alpha + \beta > 1$, we establish in Theorem \ref{abRK} a general upper bound for $\norm{T^n(I-T)^k}$ that directly improves and generalizes the estimates from \cite[Corollary 3.3, Case 4]{MahRue}. Specifically, for $k > \alpha - 1$, we prove that
	$$
	\norm{T^n(I-T)^k} = O(n^{\frac{\alpha+(\beta-1)(k+1)}{\alpha}}).
	$$
	
	By setting $\beta = 0$, these results provide, as a special case, refined estimates for $\alpha$-Ritt operators (see Corollary \ref{CorEstalphaRitt}).

	Section 5 focuses on $\beta$-Kreiss operators, i.e., $(0,\beta)$-RK operators. To analyze these operators, we introduce a generalization of classical Ces\`aro boundedness, called $C_\gamma^{(\beta)}$-boundedness, defined by the growth condition $\|S_n^{\gamma}(T)\| = O(n^{\gamma +\beta - 1})$ for $\gamma > 0$ and $\beta \ge 1$, where $S_n^{\gamma}(T) := \sum_{j=0}^n A_{n-j}^{\gamma-1} T^j$ denotes the Ces\`aro sums of order $\gamma$. We establish a characterization of $\beta$-Kreiss operators in terms of the uniform boundedness of these generalized Ces\`aro means for $\gamma > 1$. Next, we introduce the property of uniform $\beta$-Kreiss boundedness and show its equivalence to the uniform $C_1^{(\beta)}$-boundedness of $\lambda T$ for $\lambda \in \mathbb{T}$. For the special class of positive operators acting on Banach lattices, we show that $\beta$-Kreiss boundedness, uniform $\beta$-Kreiss boundedness, and $C_1^{(\beta)}$-boundedness are all equivalent, extending a classical result from the case $\beta=1$. Finally, adapting a construction from \cite{BonMul}, we exhibit an explicit example of a $\beta$-Kreiss operator on a Hilbert space whose powers satisfy the lower bound $\|T^n\| \ge c \, n^{\beta}$ for some constant $c > 0$ and for all $n \ge 1$. This proves that for any $\beta > 1$, the power growth estimate $\|T^n\| = O(n^{\beta})$ is sharp, a behavior that stands in sharp contrast to the classical case $\beta = 1$ which implies $\|T^n\| = o(n)$ in Hilbert spaces.

	Section 6 is dedicated to examples. We first study an example by Nevanlinna \cite[Example 5]{Nev} to show that the estimates in \eqref{NevEsti} are sharp for all $k \in \mathbb{N} \cup \{0\}$. Next, for $m\in \N$, the multiplication operator $M_z$ on a specific Banach space of analytic functions shows that an $m$-Kreiss operator is not necessarily $C_1^{(m)}$-bounded providing also the necessity of the $\log(n+1)$ factor in the estimate $\norm{S_n^{m}(T)} = O(n^{m}\log(n+1))$ for $m$-Kreiss operator. Next, for $m$-Ritt operators, we use matrices from \cite{BorSpi2} to show that the rate $O(n^{m-k-1})$ is sharp when $k \le m-2$. We then use a larger version of the Assani matrix to prove that Cesàro-$m$ boundedness does not imply $m$-Kreiss boundedness. Furthermore, an $\ell^2$-direct sum of finite-dimensional blocks, constructed from the examples of Borovykh and Spijker \cite{BorSpi2}, gives an $(m,1)$-RK operator with the maximal growth rate $O\big(n^m \log(n+1)\big)$. Finally, we generalize the Tomilov--Zem\'anek matrix construction. We prove that for $N\in \N$, an operator $T$ is power bounded if and only if a triangular matrix $\mathcal{T}_{N+1}$ built from $T$ is  Cesàro-$N$ bounded. This provides Cesàro-$N$ bounded operators with power growth $n^N$ such that $\Big(n^{-N}S^{N}_n\big(\mathcal{T}_{N+1}\big)\Big)_{n\in \N}$ converges strongly, while $\Big(n^{-N}S^{N}_n\big(\mathcal{T}_{N+1}\big)\Big)_{n\in \N}$ does not.
	
	Finally, Section 7 concludes the paper by providing a slight generalization of Theorem \ref{Ritt} and discussing several open questions. To place our contributions within the context of the existing literature, the known bounds and the new estimates established in this paper are summarized in Table 1
	\begin{table}[h]
		\centering
		\renewcommand{\arraystretch}{3} 
		\begin{tabular}{|c|c|c|c|c|c|}
			\hline
			&\makecell{ Conditions \\ satisfied by T} & Space &\makecell{ Growth rate of \\ $\norm{T^n(I-T)^k}$} &\makecell{Range of \\ integers $k$} &  Reference \\ \hline
			\multirow{4}{*}{$\alpha \ge 1$ } & \multirow{2}{*}{ \makecell{Power-bounded\\ and $\alpha$-Ritt} } & Banach & $\displaystyle O\Big(\frac{log(n+1)^k}{n^{k/\alpha}}\Big)$ &\multirow{4}{*}{$k \ge 0$  } &\multirow{2}{*}{\cite{Sei}} \\ \cdashline{3-4}
			&  & Hilbert & $\displaystyle O\Big(\frac{1}{n^{k/\alpha}}\Big)$ & &\\ \cline{2-4}\cline{6-6}
			& \makecell{Strongly Kreiss \\ and $\alpha$-Ritt} & Hilbert& $\displaystyle O\Big(\frac{log(n+1)^{\kappa}}{n^{1/\alpha}}\Big)$ & & Section \ref{SecStrKreiss} \\ \cline{2-4}\cline{6-6}
			& \makecell{Kreiss bounded \\and $\alpha$-Ritt}  & Banach & $\displaystyle O\Big(n^{\frac{\alpha-k-1}{\alpha}}\Big)$  &&  \makecell{\cite{Nev}, \\ Section \ref{rateofdecaygeneral}} \\ \hline
			$1\le \beta < \alpha$ & \makecell{$\beta$-Kreiss bounded\\and $\alpha$-Ritt}  & Banach & $\displaystyle O\Big(n^{\beta\frac{\alpha-k-1}{\alpha}}\Big)$  &$k < \alpha - 1$ & Section \ref{SecBetaKreiss}  \\ \hline
			\makecell{$0<\beta <1$}&  \multirow{4}{*}{\makecell{$(\alpha,\beta)$-RK \\ $\alpha+\beta >1$}} &  \multirow{4}{*}{Banach} & $\displaystyle O\Big(n^{\frac{\alpha+(\beta-1)(k+1)}{\alpha}}\Big)$  &$k > \alpha - 1$ &  \multirow{4}{*}{Section \ref{SectAlphaBetaRK}}  \\ \cdashline{1-1}\cdashline{4-5}
			$\beta \ge 0$, $\alpha> 1$&  & & $\displaystyle O\Big(n^{\alpha+\beta - k - 1}\Big)$  &$k < \alpha - 1$ &   \\ \cdashline{1-1}\cdashline{4-5}
			$\beta \ge 0$, $\alpha \in \N$&       &  & $\displaystyle O\big(n^{\beta}\log(n+1)\big)$  &$k = \alpha - 1$ &  \\ \cdashline{1-1}\cdashline{4-5}
			\makecell{$\beta \ge 0$ \\ $\alpha \ge 0, \alpha \notin \N  $}&   &  & $\displaystyle O(n^{\beta})$  &$k = \lfloor\alpha \rfloor$ &  \\ \hline
			$\alpha \ge 1$ & $\alpha$-Ritt  & Banach & $\displaystyle O\Big(n^{\frac{\alpha-k-1}{\alpha}}\Big)$  &$k > \alpha-1$ &  Section \ref{SectAlphaBetaRK}\\ \hline
		\end{tabular}
		\caption{Comparison of growth rates for $\|T^n(I-T)^k\|$}
		\label{tabgrowth_rates}
	\end{table}

	The following notation will be used throughout the paper. $X$ denotes a complex Banach space, and $B(X)$ refers to the algebra of all bounded linear operators on $X$. The identity operator is denoted by $I$. For any $T \in B(X)$, the spectrum of $T$ is denoted by $\sigma(T)$. The resolvent set is defined as $\rho(T) := \mathbb{C} \setminus \sigma(T)$, and for any $\lambda \in \rho(T)$, the resolvent operator is given by 
	\[
	R(\lambda, T) := (\lambda I - T)^{-1}.
	\]
	
	We denote the open unit disk by $\mathbb{D} = \{ z \in \mathbb{C} : |z| < 1 \}$ and its boundary by $\mathbb{T}$. An operator $T$ is said to be power-bounded if there exists a constant $M > 0$ such that $\sup_{n \ge 0} \|T^n\| \le M$.
	
	Regarding asymptotic comparisons, for two sequences of non-negative real numbers $(a_n)_{n \ge 1}$ and $(b_n)_{n \ge 1}$, we use the following notation:
	\begin{itemize}[-]
		\item $a_n = O(b_n)$ means that there exists a constant $C > 0$ such that $a_n \le C b_n$ for all $n \ge 1$.
		\item $a_n = o(b_n)$ means that $\lim_{n \to \infty} \frac{a_n}{b_n} = 0$
	\end{itemize}
	Other notations will be introduced throughout the text as they become necessary.
	
	This section ends with a result (see \cite[Lemma 3.1]{MahRue}) which will be used throughout the article.
	
	\begin{lemma}\label{LemmaIntRes}
		We have
		\[
		\int_{-\pi}^{\pi} \frac{1}{|re^{it}-1|^{\gamma}} \, dt \le 
		\begin{cases}
			C_{\gamma}, & r>1, \; 0<\gamma<1,  \\
			\frac{C_{\gamma}}{(r-1)^{\gamma - 1}}, \; & r>1, \gamma >1. \\

		\end{cases}
		\]
		Moreover,
		\[
		\int_{-\pi}^{\pi} \frac{1}{|re^{it}-1|}\, dt \le 
		C_{1} \log \bigl( (r-1)^{-1} \bigr), \quad 1<r< \frac{3}{2}.
		\]
	\end{lemma}

	
	\section{$\alpha$-Ritt and strongly Kreiss bounded operators in Hilbert spaces}\label{HilbertKreissbounded}\label{SecStrKreiss}
	
	This section is devoted to establishing a sharp upper bound for the growth of the sequences $\|T^n(I-T)^k\|$ when $T$ is an $\alpha$-Ritt operator on a Hilbert space. We show that under the stronger assumption of strong Kreiss boundedness, the classical polynomial decay can be balanced against a  logarithmic growth factor.

	The following lemma establishes the control of the orbit sums for strongly Kreiss bounded operators in Hilbert spaces (see \cite[Lemma 4.8]{CCEL}).
	
	\begin{lemma}\label{cunyeisnlinco}
		Let $T$ be a strongly Kreiss bounded operator on a Hilbert space $H$. Then there exist constants $C,\kappa >0$ such that for every $N\in \mathbb{N}$,
		$$
		\sum_{n=1}^N \|T^n x\| \le C N \log(N+1)^{\kappa}
		\quad \text{and} \quad
		\sum_{n=1}^N \|(T^*)^n x\| \le C N \log(N+1)^{\kappa}.
		$$
	\end{lemma}
	
	As a direct consequence of Lemma \ref{cunyeisnlinco}, we can bound  $\norm{T^N}$. Indeed, for any unit vectors $x, y \in H$, we use the identity $N\langle T^N x,y \rangle = \sum_{n=1}^N \langle T^n x,(T^*)^{N-n}y \rangle$. The Cauchy–Schwarz inequality, together with Lemma \ref{cunyeisnlinco}, implies
	$$
	|\langle T^N x,y \rangle| \le \frac{1}{N} \Big(\sum_{n=1}^N \|T^n x\|^2 \Big)^{1/2} \Big(\sum_{n=1}^N \|(T^*)^n y\|^2 \Big)^{1/2} \le C \log(N+1)^{\kappa}.
	$$
	Taking the supremum over $\|x\| = \|y\| = 1$, we establish the following proposition.
	
	\begin{prop}\label{PropEstiTnSK}
		Let $T$ be a strongly Kreiss bounded operator on a Hilbert space $H$. Then there exist constants $C,\kappa >0$ such that for every $N\in \mathbb{N}$,
		$$
		\norm{T^N} \le C\log(N+1)^{\kappa}.
		$$
	\end{prop}

	In the sequel, we will need to deal with fractional powers of $(I-T)$. Recall that when $I-T$ is sectorial, the operator $(I-T)^{\alpha}$ is defined via the standard sectorial functional calculus (see \cite{Haase2006} for further details).
	
	For our purposes, alongside Lemma \ref{cunyeisnlinco}, we will require the following result, which follows directly from the moment inequality (see \cite[Corollary 7.2]{Haa}).
	
	\begin{lemma}\label{momentinequality}
		Let $T \in B(X)$ be such that $I-T$ is sectorial. Then for $\alpha > 1$ and $k \in \N$,
		$$
		\|T^n(I-T)^k\|
		\le C\, \|T^n\|^{\frac{\alpha-1}{\alpha}}
		\|T^n(I-T)^{k\alpha}\|^{\frac{1}{\alpha}},
		$$
		where $C$ depends only on $\alpha$ and $k$.
	\end{lemma}

	Before stating the main result of this section, we recall the definition of $\alpha$-Ritt operators
	
	\begin{df}\label{defAlphaRitt}
		We say that $T \in B(X)$ is an $\alpha$-Ritt operator for $\alpha >1$ if $\sigma(T) \subset \overline{\mathbb{D}}$ and there exists a constant $C > 0$ such that
		\begin{equation}\label{alphaRittcondition}
			\norm{R(\lambda,T)} \le \frac{C}{|\lambda-1|^{\alpha}}, \quad 1 < |\lambda| < 2.
		\end{equation}
		The smallest constant $C$ satisfying \eqref{alphaRittcondition} is denoted by $R_{\alpha}(T)$.
	\end{df}
	
	\begin{thm}\label{MainTheorem}
		Let $T$ be an $\alpha$-Ritt operator on a Hilbert space $H$.
		If $T$ is strongly Kreiss bounded, then there exists $\kappa > 0$ such that for each $k \in \N$,
		$$
		\|T^n(I-T)^k\|
		= O\Big( \frac{\log^{\kappa}(n+1)}{n^{k/\alpha}} \Big).
		$$
	\end{thm}
	
	We begin with a preliminary lemma, the proof of which can be found in \cite[Lemma 3.9]{Sei}.
	
	\begin{lemma}\label{lemmaalpha}
		Let $\alpha > 1$ and let $T \in B(X)$ be an $\alpha$-Ritt operator such that $I-T$ is sectorial. Then there exists $M > 0$, depending on $R_{\alpha}(T)$ such that
		$$
		\|(I-T)^{\alpha}R(\lambda, T)\| \le M,
		\quad 1 < |\lambda| < 2.
		$$
	\end{lemma}
	\begin{remark}
		In \cite{Sei}, the proof of Lemma \ref{lemmaalpha} is originally provided under the assumption that $T$ is power-bounded. However, the condition of power-boundedness is not used in the proof.
	\end{remark}
	
	We are now ready to prove Theorem \ref{MainTheorem}.
	
	\begin{proof}[Proof of Theorem \ref{MainTheorem}]
		Let $k \in \N$ and let $T$ be an $\alpha$-Ritt operator.
		The first step consists in showing that
		\begin{equation}\label{main-est}
			\sum_{n=0}^N \binom{n+k}{k}^2
			\|T^n (I-T)^{k\alpha} x\|^2
			\le C \sum_{n=0}^N \|T^n x\|^2,
			\quad N \in \N,
		\end{equation}
		where $C$ depends only on $k$.
		
		Let $\lambda \in \mathbb{C}$ with $|\lambda| > 1$ and define
		$$
		\varphi_{N,k}(\lambda)
		:= \lambda^k R(\lambda, T)^k
		\sum_{n=0}^N \frac{T^n}{\lambda^{n}},
		\quad N \in \N.
		$$
		A straightforward computation based on the identity
		$$
		\lambda^k R(\lambda, T)^k
		= \sum_{n=0}^{\infty}
		\binom{n+k-1}{k-1}\frac{T^n}{\lambda^{n}}
		$$
		leads to
		$$
		\varphi_{N,k}(\lambda)
		= \sum_{n=0}^\infty
		\sum_{j=0}^{\min\{n,N\}}
		\binom{n-j+k-1}{k-1}
		\frac{ T^n}{\lambda^n}.
		$$
		Applying Parseval’s identity yields
		$$
		\sum_{n=0}^N
		\Bigg(\sum_{j=0}^{n}
		\binom{n-j+k-1}{k-1}\Bigg)^2
		\frac{\|T^n x\|^2}{\rho^{2n}}
		\le \frac{1}{2\pi}
		\int_0^{2\pi}
		\|\varphi_{N,k}(\rho e^{i\theta})x\|^2
		\, d\theta,
		$$
		for every $x \in H$ and $\rho > 1$.
		
		Using
		\begin{equation*}\label{binom}
			\sum_{j=0}^m \binom{r+j}{r}
			= \binom{r+m+1}{r+1},
		\end{equation*}
		we obtain
		$$
		\sum_{j=0}^{n}
		\binom{n-j+k-1}{k-1}
		= \binom{n+k}{k}.
		$$
		It follows that
		$$
		\sum_{n=0}^N
		\binom{n+k}{k}^2
		\frac{\|T^n x\|^2}{\rho^{2n}}
		\le \frac{1}{2\pi}
		\int_0^{2\pi}
		\|\varphi_{N,k}(\rho e^{i\theta})x\|^2
		\, d\theta.
		$$
		Substituting $x$ by $(I-T)^{k\alpha}x$ and using Lemma \ref{lemmaalpha},
		together with Parseval’s identity, we obtain
		\begin{align*}
			\sum_{n=0}^N
			\frac{\binom{n+k}{k}^2}{\rho^{2n}}
			\|T^n (I-T)^{k\alpha} x\|^2
			&\le \frac{M^{2k}\rho^{2k}}{2\pi}
			\sum_{n=0}^N
			\frac{\|T^n x\|^2}{\rho^{2n}}.
		\end{align*}
		Letting $\rho \to 1^+$ yields \eqref{main-est}.
		
		For $(x,y)\in H \times H^*$ and $N \in \N$, we have
		$$
		\binom{N+k+1}{k+1}
		\langle T^N (I-T)^{k\alpha} x, y \rangle
		= \sum_{n=0}^N
		\left\langle \binom{n+k}{k} T^n (I-T)^{k\alpha} x,
		(T^*)^{N-n} y \right\rangle.
		$$
		
		By the Cauchy--Schwarz inequality, it follows
		\begin{align}\label{ineTNI-Talpha}
			|\langle T^N (I-T)^{k\alpha} x, y \rangle|^2
			\le \frac{1}{\binom{N+k+1}{k+1}^2}
			\Bigg(
			\sum_{n=0}^N
			\binom{n+k}{k}^2
			\|T^n (I-T)^{k\alpha}x\|^2
			\Bigg)
			\Bigg(
			\sum_{n=0}^N
			\|(T^*)^{N-n}y\|^2
			\Bigg).
		\end{align}
		
		Now assume that $T$ is strongly Kreiss bounded.
		By Lemma \ref{cunyeisnlinco} and \eqref{main-est},
		there exists $\kappa > 0$ such that
		\begin{equation*}\label{eq:3.19-new}
			\sum_{n=0}^N
			\binom{n+k}{k}^2
			\|T^n (I-T)^{k\alpha} x\|^2
			\le C N \log^{\kappa}(N+1) \|x\|^2,
		\end{equation*}
		and
		$$
		\sum_{n=0}^N\|(T^*)^{N-n}y\|^2
		\le CN \log^{\kappa}(N+1).
		$$
		Combining these estimates with \eqref{ineTNI-Talpha} and taking the supremum over $\|x\| = \|y\| = 1$, we obtain
		$$
		\|T^N(I-T)^{k\alpha}\|
		\le \frac{C N \log^{\kappa}(N+1)}{\binom{N+k+1}{k+1}},
		\quad N \in \N.
		$$
		
		Since
		\begin{equation*}\label{estibinom}
			\frac{1}{\binom{N+k+1}{k+1}} \le \frac{(k+1)!}{N^{k+1}},
		\end{equation*}
		it follows that
		$$
		\|T^N(I-T)^{k\alpha}\|
		\le \frac{C \log^{\kappa}(N+1)}{N^k}.
		$$
		for $C$ depending only on $k$.
		Finally, applying Lemma \ref{momentinequality} and Proposition \ref{PropEstiTnSK}, we obtain
		\begin{align*}
			\|T^N (I-T)^k\|
			&\le C \|T^N\|^{(\alpha-1)/\alpha}
			\|T^N (I-T)^{k\alpha}\|^{1/\alpha} \le \frac{\log^{\kappa}(N+1)}{N^{k/\alpha}}.
		\end{align*}
	\end{proof}

	
	\section{A Detailed Review and Proof of Theorem \ref{Nevanlinna}}\label{rateofdecaygeneral}
	
	This section is devoted to providing a detailed proof of Theorem \ref{Nevanlinna}. Although this result was originally introduced by Nevanlinna \cite{Nev}, several key intermediate steps and technical estimates lack explicit details in the original text. We present a complete and self-contained proof here, both for the sake of clarity and because this foundational result will be heavily utilized in the subsequent sections of this paper.

	Let $T \in B(X)$. We recall that for $\lambda \in \rho(T)$, the series 
	$$
	\sum_{k=0}^{\infty} (\lambda - \mu)^k R(\lambda,T)^{k+1}
	$$
	is well-defined for $\mu$ belonging to the set $\mathcal{D}_{T,\lambda} := \big\{z \in \mathbb{C} : |\lambda - z| < \|R(\lambda,T)\|^{-1} \big\}$. In this case, the resolvent satisfies the identity
	\begin{equation}\label{eqKR}
		R(\mu, T) = \sum_{k=0}^{\infty} (\lambda - \mu)^k R(\lambda,T)^{k+1}, \quad \mu \in \mathcal{D}_{T,\lambda}.
	\end{equation}
	
	We begin with a useful preliminary lemma (see \cite[Lemma, p.~20]{Nev}).
	
	\begin{lemma}\label{lemNev}
		Let $\alpha > 1$ and let $T \in B(X)$ be an $\alpha$-Ritt operator. Then, for each $c \in \big(0, \frac{1}{2R_{\alpha}(T)(2\pi)^{\alpha}}\big)$, the closed curve $\Gamma$ defined by 
		$$
		\Gamma(\theta) := \Gamma_{c,\alpha}(\theta) := e^{i\theta - c|\theta|^{\alpha}}, \quad \theta \in [-\pi, \pi],
		$$
		satisfies the estimate
		\begin{equation}\label{Resineup}
			\|R(\lambda,T)\| \le \frac{C}{|\lambda-1|^{\alpha}}, \quad \lambda \in \Omega . 
		\end{equation}
		Here $\Omega := \{\lambda \in \operatorname{ext}(\Gamma), \; |\lambda|\le 2 \}$ and $\operatorname{ext}(\Gamma)$ denotes the exterior of the closed curve $\Gamma$ and the constant $C$ depends on $c$, $\alpha$ and $R_{\alpha}(T)$.
	\end{lemma}
	\begin{center}
		\begin{tikzpicture}[scale=2.4]
			\draw[->] (-1.1,0) -- (1.1,0) node[right] {$x$};
			\draw[->] (0,-1.1) -- (0,1.1) node[above] {$y$};
			\foreach \x/\xtext in {-1} {
				\draw (\x,0.05cm) -- (\x,-0.05cm) node[below] {$\xtext\phantom{-}\strut$};
			}
			\foreach \x/\xtext in {1} {
				\draw (\x,0.05cm) -- (\x,-0.05cm) node[below] {$\xtext\strut$};
			}
			\foreach \y in {-1,1} {
				\draw (0.05cm,\y) -- (-0.05cm,\y) node[left] {$\y\strut$};
			}
			\node[below left=0.1cm] at (-0,0) {$0\strut$};
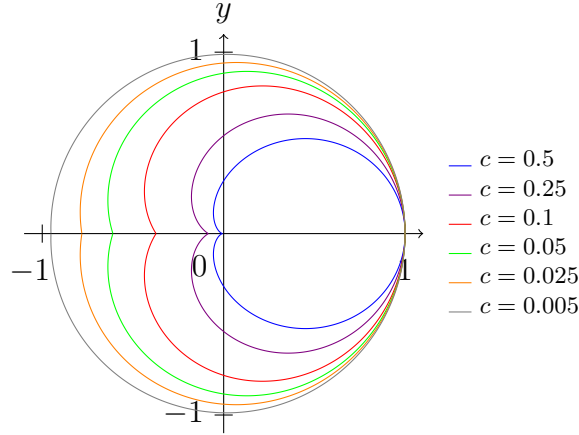
			\node[anchor=south] at (0, -1.4) { Curves $\Gamma_{c,\alpha}$ for $\alpha = 2$ };
			\begin{scope}
				\clip (-1.2,-1.2) rectangle (1.2,1.2);
				\draw[color=blue,domain=-3.14:3.14,samples=200,smooth]
				plot ({exp(-0.5*abs(\x)^2)*cos(deg(\x))}, {exp(-0.5*abs(\x)^2)*sin(deg(\x))});
				\draw[color=violet,domain=-3.14:3.14,samples=200,smooth]
				plot ({exp(-0.25*abs(\x)^2)*cos(deg(\x))}, {exp(-0.25*abs(\x)^2)*sin(deg(\x))});
				\draw[color=red,domain=-3.14:3.14,samples=200,smooth]
				plot ({exp(-0.1*abs(\x)^2)*cos(deg(\x))}, {exp(-0.1*abs(\x)^2)*sin(deg(\x))});
				\draw[color=green,domain=-3.14:3.14,samples=200,smooth]
				plot ({exp(-0.05*abs(\x)^2)*cos(deg(\x))}, {exp(-0.05*abs(\x)^2)*sin(deg(\x))});
				\draw[color=orange,domain=-3.14:3.14,samples=200,smooth]
				plot ({exp(-0.025*abs(\x)^2)*cos(deg(\x))}, {exp(-0.025*abs(\x)^2)*sin(deg(\x))});
				\draw[color=gray,domain=-3.14:3.14,samples=200,smooth]
				plot ({exp(-0.005*abs(\x)^2)*cos(deg(\x))}, {exp(-0.005*abs(\x)^2)*sin(deg(\x))});
			\end{scope}
			\node [ fill=white, font=\scriptsize] at (1.6, 0) {
				\begin{tabular}{ll}
					\textcolor{blue}{\rule{0.3cm}{0.4pt}}  $c = 0.5$ \\
					\textcolor{violet}{\rule{0.3cm}{0.4pt}}  $c = 0.25$ \\
					\textcolor{red}{\rule{0.3cm}{0.4pt}}  $c = 0.1$ \\
					\textcolor{green}{\rule{0.3cm}{0.4pt}}  $c = 0.05$ \\
					\textcolor{orange}{\rule{0.3cm}{0.4pt}}  $c = 0.025$ \\
					\textcolor{gray}{\rule{0.3cm}{0.4pt}}  $c = 0.005$ \\
				\end{tabular}
			};
		\end{tikzpicture}
	\end{center}

	\begin{proof}
		First from the preceding discussion (see \cite[Proposition 2.2]{MahRue}, for  every $\theta \in [-\pi,\pi]\backslash\{0\}$, $\norm{R(e^{i\theta},T)} \le \frac{R_{\alpha}(T)}{|e^{i\theta}-1|^{\alpha}}$. Furthermore for $\theta \in [-\pi,\pi]\backslash\{0\}$, $R(e^{i\theta} - z, T)$ is well-defined as long as $e^{i\theta}-z \in \mathcal{D}_{T,e^{i\theta}}$. In this case, using \eqref{eqKR} and the $\alpha$-Ritt condition, we have
		\begin{align*}
			\|R(e^{i\theta}-z,T)\| 
			&= \left\| \sum_{k=0}^{\infty} (e^{i\theta}-z-e^{i\theta})^k R(e^{i\theta},T)^{k+1} \right\| \\
			&\le \frac{R_{\alpha}(T)}{|e^{i\theta}-1|^{\alpha}} \sum_{k=0}^{\infty} \frac{R_{\alpha}(T)^k |z|^k}{|e^{i\theta}-1|^{k\alpha}}.
		\end{align*}
		By choosing $z \in \mathbb{C}$ such that $|z| \le \frac{|e^{i\theta}-1|^{\alpha}}{2R_{\alpha}(T)}$, we ensure that $e^{i\theta}-z \in \mathcal{D}_{T,e^{i\theta}}$ and
		\begin{equation}\label{inelemN}
			\|R(e^{i\theta}-z,T)\| \le \frac{2R_{\alpha}(T)}{|e^{i\theta}-1|^{\alpha}}.
		\end{equation}
		Furthermore, since $1-e^{-x} \le x$ for every real number $x$, we have
		$$
		1-e^{-c|\theta|^{\alpha}} \le c|\theta|^{\alpha}, \quad \theta \in [-\pi, \pi].
		$$ 
		Utilizing the standard inequality
		\begin{equation}\label{inesin}
			|\theta| \le \frac{\pi}{2} \times 2\left|\sin\Big(\frac{\theta}{2}\Big)\right| = \frac{\pi}{2}|1-e^{i\theta}|, \quad \theta \in [-\pi,\pi], 
		\end{equation}
		it follows that by setting $z = e^{i\theta} - \Gamma(\theta)$, we obtain 
		\begin{equation*}\label{inezalp}
			|z| = 1-e^{-c|\theta|^{\alpha}} \le c \left(\frac{\pi}{2}|1-e^{i\theta}|\right)^{\alpha} \le \frac{|1-e^{i\theta}|^{\alpha}}{2R_{\alpha}(T)}. 
		\end{equation*}
		Since $e^{i\theta} - z = e^{i\theta-c|\theta|^{\alpha}}$, inequality \eqref{inelemN} implies that
		$$
		\|R(\Gamma(\theta),T)\| \le \frac{2R_{\alpha}(T)}{|e^{i\theta}-1|^{\alpha}}, \quad \theta \in [-\pi,\pi].
		$$
		We note that since $\alpha > 1$, 
		\begin{equation*}\label{esteitheta}
			|1-e^{i\theta - c|\theta|^{\alpha}}| \le \Big(1+\frac{c\pi^{\alpha}}{2}\Big)|1-e^{i\theta}|, \quad \theta \in [-\pi, \pi],
		\end{equation*}
		which yields 
		\begin{equation}\label{EstiuGamma}
			\|R(\Gamma(\theta),T)\| \le \frac{2R_{\alpha}(T)}{\Big(1+\frac{c\pi^{\alpha}}{2}\Big)|e^{i\theta-c|\theta|^{\alpha}}-1|^{\alpha}} = \frac{2R_{\alpha}(T)}{\Big(1+\frac{c\pi^{\alpha}}{2}\Big)|\Gamma(\theta) - 1|^{\alpha}}, \quad \theta \in [-\pi,\pi].
		\end{equation}
		Finally, consider the function
		$$
		\begin{array}{ccccc}
			u & : & \Omega := \{\lambda \in \operatorname{ext}(\Gamma), \; |\lambda|\le 2 \} & \to & \mathbb{R}_+ \\
			& & \lambda & \mapsto & |\lambda-1|^{\alpha}\|R(\lambda, T)\|. \\
		\end{array}
		$$
		Since $\lambda \mapsto (\lambda-1)^{\alpha} R(\lambda,T)$ is analytic on  $\Omega$  (where we use an appropriate determination of logarithm), then by \cite[Theorem 3.13.1]{HilPhi}, $u$ is subharmonic on $\Omega$. Moreover $u$ is bounded on $\{\lambda \in \mathbb{C} : |\lambda| = 2\}$ as well as on $\{ \Gamma(\theta) : \theta \in [-\pi,\pi] \}$ by \eqref{EstiuGamma}, the required estimate \eqref{Resineup} follows directly from the maximum principle.
	\end{proof}

	To facilitate the presentation of the main proof, we establish two technical lemmas.
	
	\begin{lemma}\label{lem_gamma_estimates}
		Let $a > 0$, $b \in \mathbb{R}$, and $d > 0$. Then 
		\begin{equation}\label{esti2fungamma}
			\int_{ \big(\frac{d}{n}\big)^{1/a} }^{\infty} e^{-n d \theta^{a}} \theta^{b} \, d\theta
			\le C n^{-\frac{1+b}{a}}, \quad n \in \mathbb{N},
		\end{equation}
		where $C$ depends only on $a, b$, and $d$. 
	\end{lemma}
	
	\begin{proof}
		Performing the change of variables $t = n d \theta^{a}$, we obtain
		$$
		\theta = \Big(\frac{t}{n d}\Big)^{1/a}, \qquad 
		d\theta = \frac{1}{a}(n d)^{-1/a} t^{\frac{1}{a}-1} \, dt.
		$$
		Hence,
		\begin{align*}
			\int_{ \big(\frac{d}{n}\big)^{1/a} }^{\infty} e^{-n d \theta^{a}} \theta^{b} \, d\theta
			&= \frac{1}{a}(n d)^{-\frac{1+b}{a}} \int_{d^{2}}^{\infty} t^{\frac{1+b}{a}-1} e^{-t} \, dt,
		\end{align*}
		and the estimate \eqref{esti2fungamma} follows.
	\end{proof}
	
	\begin{remark}
		If $b > -1$, then $\frac{1+b}{a}-1 > -1$. Since $\int_{0}^{\infty} t^{s} e^{-t} \, dt < \infty$ for any $s > -1$, we obtain 
		\begin{equation}\label{esti3fungamma}
			\int_{0}^{\infty} e^{-n d \theta^{a}} \theta^{b} \, d\theta
			\le C n^{-\frac{1+b}{a}}, \quad n \in \mathbb{N}, \, a > 0, \, b > -1, 
		\end{equation}
		where $C$ depends on $a, b$, and $d$.
	\end{remark}
	
	\begin{lemma}
		Let $r > 0$. Then 
		\begin{equation}\label{ine12eitheta}
			|re^{i\theta} - 1| \ge \frac{1}{2}|e^{i\theta} - 1|, \quad \theta \in [-\pi, \pi].
		\end{equation}
	\end{lemma}

	\begin{proof}
		Let $\theta \in [-\pi, \pi]$. We observe that
		$$
		|r e^{i\theta}-1|^{2} = r^{2}-2r\cos\theta+1,
		$$
		whose minimum over $r \ge 0$ occurs at $r = \max(\cos\theta, 0)$. 
		
		If $\cos\theta \le 0$, then $|r e^{i\theta}-1| \ge 1$. Since $|e^{i\theta} - 1| = 2|\sin(\tfrac{\theta}{2})| \le 2$, estimate \eqref{ine12eitheta} immediately follows.
		
		If $\cos\theta > 0$, then 
		$$
		|r e^{i\theta} - 1| \ge |\cos(\theta) e^{i\theta} - 1| = |\sin(\theta)|.
		$$
		Using the identity $\sin\theta = 2\sin(\tfrac{\theta}{2})\cos(\tfrac{\theta}{2})$ and noting that $\cos(\tfrac{\theta}{2}) \ge \frac{\sqrt{2}}{2} \ge \frac{1}{2}$ whenever $\cos\theta > 0$, we obtain
		$$
		|re^{i\theta} - 1| \ge |\sin\theta| = 2\left|\sin\Big(\frac{\theta}{2}\Big)\right|\left|\cos\Big(\frac{\theta}{2}\Big)\right|
		\ge \left|\sin\Big(\frac{\theta}{2}\Big)\right| = \frac{1}{2}|e^{i\theta} - 1|,
		$$
		which establishes inequality \eqref{ine12eitheta}.
	\end{proof}

	\begin{proof}[Proof of Theorem \ref{Nevanlinna}]
		For any integer $n > 1$, let $\Gamma_n := \big(1+\frac{1}{n}\big)\Gamma$, where $\Gamma$ is the curve defined in Lemma \ref{lemNev}. Then $\Gamma_n([-\pi, \pi]) \subset \operatorname{ext}(\Gamma) \cap D(0,2)$. By the Dunford--Riesz operational calculus, we can write
		$$
		T^n(I-T)^k = \frac{1}{2i\pi} \int_{\Gamma_n} \lambda^n(1-\lambda)^k R(\lambda,T) \, d\lambda.
		$$
		Taking the norm yields
		$$
		\|T^n(I-T)^k\| \le \frac{1}{2\pi} \int_{\pi}^{\pi} \big|\Gamma_n'(\theta)\big| \big|\Gamma_n(\theta)\big|^n \big|1-\Gamma_n(\theta)\big|^k \|R(\Gamma_n(\theta),T)\| \, d\theta.
		$$
		For $\theta \in [-\pi, \pi]$, we have $|\Gamma_n(\theta)| = \big(1+\frac{1}{n}\big)e^{-c|\theta|^{\alpha}}$ and
		$$
		|\Gamma'_n(\theta)| = \Big(1+\frac{1}{n}\Big)\Big|-c\alpha \operatorname{sign}(\theta)|\theta|^{\alpha-1} + i\Big|e^{-c|\theta|^{\alpha}} \le 2(c\alpha \pi^{\alpha-1} + 1).
		$$
		Using \eqref{Resineup} along with the inequality $\big(1+\frac{1}{n}\Big)^n \le e$ for all $n \in \mathbb{N}$, we find
		\begin{align*}
			\|T^n(I-T)^k\| \le C \int_{-\pi}^{\pi} e^{-nc|\theta|^{\alpha}} |1-\Gamma_n(\theta)|^k \|R(\Gamma_n(\theta),T)\| \, d\theta,
		\end{align*}
		where $C$ depends on $c$ and $\alpha$. 
		
		For integers $n > 1$, we define $\theta_n := \big(\frac{c}{n}\big)^{1/\alpha}$. For any $\theta$ satisfying $0 \le |\theta| \le \theta_n$, we observe that
		$$
		|\Gamma_n(\theta)| \ge |\Gamma_n(\theta_n)| = \Big(1+\frac1n\Big)e^{-c^2/n} \ge \Big(1+\frac1n\Big)\Big(1-\frac{c^2}{n} \Big) = 1 + \frac{1-c^2}{n} - \frac{c^2}{n^2} \ge 1 + \frac{1-2c^2}{n}.
		$$
		Thus, assuming $c < \frac{1}{\sqrt{2}}$, we ensure that $|\Gamma_n(\theta)| > 1$ for all $0 \le |\theta| \le \theta_n$.
		
		We now partition the interval of integration into two regions:
		\begin{align}\label{defI1I2}
			\nonumber \int_{-\pi}^{\pi} e^{-nc|\theta|^{\alpha}} |1-\Gamma_n(\theta)|^k \|R(\Gamma_n(\theta),T)\| \, d\theta
			&= \int_{0 \le |\theta| \le \theta_n} \cdots \, d\theta + \int_{\theta_n \le |\theta| \le \pi} \cdots \, d\theta \\
			&=: I_1 + I_2.
		\end{align}
		On the interval $0 \le |\theta| \le \theta_n$, we exploit the hypothesis that $T$ is Kreiss bounded (since $|\Gamma_n(\theta)|> 1$ there). On the remaining interval $\theta_n \le |\theta| \le \pi$, we leverage the $\alpha$-Ritt condition. This yields the bounds
		$$
		I_1 \le \int_{0 \le |\theta| \le \theta_n} e^{-nc|\theta|^{\alpha}} \frac{|1-\Gamma_n(\theta)|^{k}}{|\Gamma_{n}(\theta)|-1} \, d\theta \quad \text{and} \quad I_2 \le \int_{\theta_n \le |\theta| \le \pi} e^{-nc|\theta|^{\alpha}} |1-\Gamma_n(\theta)|^{k-\alpha} \, d\theta.
		$$
		For any $\theta$ such that $0 < |\theta| \le \pi$, we can write
		\begin{align*}
			\Big|\Big(1+\frac{1}{n}\Big)e^{-c|\theta|^{\alpha}+i\theta} - 1\Big| 
			&\le |e^{i\theta}-1| + \Big|\Big(1+\frac{1}{n}\Big)e^{-c|\theta|^{\alpha}} - 1 \Big| \\
			&\le |\theta| + |1-e^{-c|\theta|^{\alpha}}| + \frac{e^{-c|\theta|^{\alpha}}}{n} \\
			&\le |\theta| + c|\theta|^{\alpha} + \frac{1}{n}. 
		\end{align*}
		Consequently, we obtain the estimate
		\begin{equation}\label{estGamma_nmin1}
			\big|1-\Gamma_n(\theta)\big| \le \left\{
			\begin{array}{ll}
				C\Big(|\theta| + \frac1n\Big) & \text{if } 0 \le |\theta| \le \theta_n, \\
				C|\theta| & \text{if } \theta_n \le |\theta| \le \pi,
			\end{array}
			\right.
		\end{equation}
		where $C$ depends on $\alpha$ and $c$.
		
		To bound $I_1$, we use $e^{-nc|\theta|^{\alpha}} \le 1$ and the lower bound $|\Gamma_n(\theta)| - 1 \ge \frac{1-2c^2}{n}$ valid for $0 \le |\theta| \le \theta_n$. It follows that
		\begin{align*} 
			I_1 &\le \frac{C^k n}{1-2c^2} \int_{0 \le |\theta| \le \theta_n} \Big(|\theta|+\frac{1}{n}\Big)^k \, d\theta 
			= \frac{2C^k n}{(1-2c^2)(k+1)} \left[ \Big(\frac{c}{n^{1/\alpha}} + \frac{1}{n}\Big)^{k+1} - \frac{1}{n^{k+1}} \right] \\
			&\le C' n^{1-\frac{{k+1}}{\alpha}} = C' n^{\frac{\alpha - k - 1}{\alpha}},
		\end{align*}
		where $C'$ depends on $\alpha,c$ and $k$.
		
		To bound $I_2$, we apply inequalities \eqref{ine12eitheta} and \eqref{inesin} to obtain
		\begin{equation}\label{ineremoveass}
			|1-\Gamma_n(\theta)| \ge \frac{1}{2}|e^{i\theta}-1| \ge \frac{1}{2} \cdot \frac{2}{\pi}|\theta| = \frac{1}{\pi}|\theta|, \quad \theta \in [-\pi,\pi].
		\end{equation}
		By combining \eqref{estGamma_nmin1} (when $k-\alpha \ge 0$) and \eqref{ineremoveass} (when $k-\alpha < 0$), we deduce that
		$$
		\int_{\theta_n \le \theta \le \pi} e^{-nc\theta^{\alpha}} |1-\Gamma_n(\theta)|^{k-\alpha} \, d\theta \le C'' \int_{\theta_n \le \theta \le \pi} e^{-nc\theta^{\alpha}} \theta^{k-\alpha} \, d\theta,
		$$
		where $C''$ depends on $\alpha,c$ and $k$.
		Finally, using \eqref{esti2fungamma}, we conclude that
		$$
		I_2 \le 2C'' \int_{\big(\frac{c}{n}\big)^{1/\alpha}}^{\infty} e^{-nc\theta^{\alpha}} \theta^{k-\alpha} \, d\theta \le M n^{-\frac{1+k-\alpha}{\alpha}} = M n^{\frac{\alpha-k-1}{\alpha}}.
		$$
		This completes the proof.
	\end{proof}

	We conclude this section with a proposition stating that in the case $\alpha = 2$, the estimates given in \eqref{NevEsti} are sharp. The proof is provided in \cite[Example 5]{Nev} for $k \in \{0, 1\}$. We provide further details in Section \ref{SectionExempleNev}, including the cases $k \ge 2$.
	
	\begin{prop}\label{PropNevExample}
		There exist a Banach space $X$ and an operator $T \in B(X)$ which is Kreiss bounded and $2$-Ritt, such that for some constant $C > 1$ and all $n \ge 1$:
		\[
		\frac{1}{C} n^{\frac{1-k}{2}} \le \norm{T^n(I-T)^k} \le C n^{\frac{1-k}{2}}.
		\]
	\end{prop}

	\section{$(\alpha,\beta)$-RK Operators and Power Growth Estimates}\label{SectAlphaBetaRK}

	In this section, we study a class of bounded linear operators that generalizes both the Ritt and Kreiss boundedness conditions. This $(\alpha,\beta)$-RK condition was originally introduced by Borovykh and Spijker \cite[Section 4]{BorSpi} in the framework of numerical stability for initial value problems, and has since been further investigated by Mahillo and Rueda \cite{MahRue}. Our primary objective is to establish sharp bounds on the growth of the sequence $\norm{T^n(I-T)^k}$ for $k \in \mathbb{N} \cup \{0\}$ by using techniques employed in the proof of Theorem \ref{Nevanlinna}.
	
	\begin{df}
		We say that $T \in B(X)$ is an $(\alpha,\beta)$-RK operator for $\alpha,\beta > 0$ if $\sigma(T) \subset \overline{\mathbb{D}}$ and there exists a constant $C > 0$ such that
		\begin{equation}\label{abcondition}
			\norm{R(\lambda,T)} \le \frac{C}{|\lambda-1|^{\alpha}(|\lambda|-1)^{\beta}}, \quad 1 < |\lambda| < 2.
		\end{equation}
		In the limiting case where $\beta = 0$, $T$ reduces to an $\alpha$-Ritt operator, whereas for $\alpha = 0$, $T$ is said to be $\beta$-Kreiss bounded. We denote by $R_{\alpha, \beta}(T)$ the smallest constant $C$ satisfying \eqref{abcondition}.
	\end{df}
	
	Let $T$ be an $(\alpha,\beta)$-RK operator. If $\alpha +\beta < 1$, then $\norm{T^n} = O(e^{-\omega n})$ for some $\omega > 0$ (see \cite[Corollary 3.3. case 1]{MahRue}. If $\alpha + \beta = 1$, then $T$ is a Ritt operator. Furthermore, provided that $\beta \ge 1$, then by \cite[Corollary 3.3.]{MahRue}, $\norm{T^n} = O(G_{\alpha,\beta}(n))$, where the auxiliary function $G_{\alpha,\beta}(n)$ is defined as
	\begin{equation*}
		G_{\alpha,\beta}(n) = 
		\begin{cases}
			n^{\beta}, & \text{if } 0 \le \alpha < 1, \\
			n^{\beta}\log(n+1), & \text{if } \alpha = 1, \\
			n^{\beta+\alpha -1}, & \text{if } \alpha > 1.
		\end{cases}
	\end{equation*}
	Since $\beta$-Kreiss boundedness constitutes a less restrictive condition than $(\alpha, \beta)$-RK boundedness, one cannot generally expect an estimate sharper than $\norm{T^n(I-T)^k} = O(G_{\alpha, \beta}(n))$ for $k \in \mathbb{N} \cup \{0\}$. Motivated by these observations, the subsequent analysis focuses on the non-trivial regime where $\alpha + \beta > 1$. We first state the following result: 
	
	\begin{thm}\label{abRK}
		Let $\alpha>0$ such that $\alpha + \beta > 1$. Let $T \in B(X)$ be an $(\alpha,\beta)$-RK operator. For any $k \in \mathbb{N} \cup \{0\}$, the following asymptotic estimates hold:
		\begin{enumerate}
			\item If $k < \alpha - 1$, then 
			$$
			\norm{T^n(I-T)^k} = O\Big(n^{\alpha + \beta - k -1}\Big).
			$$
			\item If $k = \lfloor \alpha - 1 \rfloor$, then 
			$$
			\norm{T^n(I-T)^k} = \begin{cases}
				O\Big(n^{\beta}\log(n+1)\Big), &\text{if } \alpha \in \N, \\
				O\Big(n^{\beta}\Big), &\text{if } \alpha \notin \N.
			\end{cases}
			$$
		\end{enumerate}
	\end{thm}
	\begin{proof}
		By the Riesz-Dunford calculus, for $\rho > 1$, we have
		$$
		T^n(I-T)^k = \frac{1}{2i\pi} \int_{0}^{2\pi} \rho^{n+1} e^{i(n+1)t} (1-\rho e^{it})^k R(\rho e^{it},T) \, dt, \quad n \in \mathbb{N}.
		$$
		Applying the $(\alpha,\beta)$-RK condition \eqref{abcondition} yields
		$$
		\norm{T^n(I-T)^k} \le R_{\alpha,\beta}(T)\frac{\rho^{n+1}}{2\pi} \int_{0}^{2\pi} \frac{1}{|\rho e^{it}-1|^{\alpha - k}(\rho-1)^{\beta}} \, dt.
		$$
		Choosing $\rho = 1 + \frac{1}{n}$ and using Lemma \ref{LemmaIntRes},
		we obtain the desired result.
	\end{proof}
	
	For the regime $\beta < 1$, we are able to establish explicit estimates for $\norm{T^n(I-T)^k}$ when the integers satisfy $k > \alpha -1$, exploiting the fact that $T$ is of Ritt-type in this case. More precisely, according to \cite[Proposition 2.2]{MahRue}, we have the following result:
	\begin{lemma}\label{RKres}
		Let $T$ be an $(\alpha,\beta)$-RK operator on $X$ with $\beta < 1$ and $\alpha > 0$. Then $T$ is an $\alpha_{\beta}$-Ritt operator, where $\alpha_{\beta} := \frac{\alpha}{1 -\beta}$.
	\end{lemma}
	This allows us to derive the following bounds :
	
	\begin{thm}\label{abRK}
		Let $\alpha>0$ and $0 \le \beta < 1$ such that $\alpha + \beta > 1$. Let $T \in B(X)$ be an $(\alpha,\beta)$-RK operator. For any $k > \alpha - 1$, we have :
		$$
		\norm{T^n(I-T)^k} = O\Big(n^{\frac{\alpha+(\beta-1)(k+1)}{\alpha}}\Big).
		$$
		
	\end{thm}
	
	To establish the Theorem \ref{abRK}, we require the following estimate.
	
	\begin{lemma}\label{lem_gamma_estimates2}
		If $b > -1$ and $a > 1$, then there exists a constant $C$ depending only on $a,b$ and $d$ such that
		\begin{equation}\label{esti1fungamma_sec3}
			\int_{0}^{\infty} e^{-n d \theta^{a}} \Big(\theta+\frac1n\Big)^{b} \, d\theta \le C n^{-\frac{1+b}{a}}, \quad n \in \mathbb{N},
		\end{equation}
		where $C$ depends only on $a, b$, and $d$.
	\end{lemma}

	\begin{proof}
		Applying the change of variables $t = n d \theta^{a}$, the integral can be rewritten as
		$$
		\int_{0}^{\infty} e^{-n d \theta^{a}} \Big(\theta+\frac1n\Big)^{b} \, d\theta
		= \frac{1}{a} (n d)^{-\frac{b+1}{a}} \int_{0}^{\infty} e^{-t} t^{\frac{b+1}{a} - 1} \left(1 + n^{-\left(1 - \frac{1}{a}\right)} d^{\frac{1}{a}} t^{-\frac{1}{a}}\right)^{b} \, dt.
		$$
		Next, we use the property $(1+x)^b \le C_b(1+x^b)$ for $x \ge 0$, where $C_b := \max(1,2^{b-1})$. Furthermore, since $a>1$ and $b>0$, we have $b\left(1 - \frac{1}{a}\right) > 0$, which immediately implies that $n^{-b\left(1 - \frac{1}{a}\right)} \le 1$ for all $n \in \mathbb{N}$. Applying these properties, we can bound the integral on the right-hand side uniformly with respect to $n$:
		\begin{align*}
			\int_{0}^{\infty} e^{-t} t^{\frac{b+1}{a} - 1} \left(1 + n^{-\left(1 - \frac{1}{a}\right)} d^{\frac{1}{a}} t^{-\frac{1}{a}}\right)^{b} \, dt 
			&\le C_b \int_{0}^{\infty} e^{-t} t^{\frac{b+1}{a} - 1} \left(1 + n^{-b\left(1 - \frac{1}{a}\right)} d^{\frac{b}{a}} t^{-\frac{b}{a}}\right) \, dt  \\
			&\le C_b \int_{0}^{\infty} e^{-t} t^{\frac{b+1}{a} - 1} \left(1 + d^{\frac{b}{a}} t^{-\frac{b}{a}}\right) \, dt \\
			&= C_b \int_{0}^{\infty} e^{-t} t^{\frac{b+1}{a} - 1} \, dt + C_b d^{\frac{b}{a}} \int_{0}^{\infty} e^{-t} t^{\frac{1}{a} - 1} \, dt.
		\end{align*}
		
		Since $\frac{b+1}{a} > 0$ and $\frac{1}{a} > 0$, the exponents satisfy $\frac{b+1}{a}-1 > -1$ and $\frac{1}{a} - 1 > -1$. Thus, both integrals on the right-hand side converge, which yields \eqref{esti1fungamma_sec3}.
	\end{proof}

	\begin{proof}[Proof of Theorem \ref{abRK}]
		Let $k \in \mathbb{N} \cup \{0\}$.
		
		Assume that $k > \alpha - 1$. Since $T$ is an $(\alpha,\beta)$-RK operator with $\beta < 1$, Lemma \ref{RKres} ensures that $T$ is an $\alpha_{\beta}$-Ritt operator with $\alpha_{\beta} = \frac{\alpha}{1 -\beta}$. Following the proof of Theorem \ref{Nevanlinna}, we define $\Gamma_n := \big(1+\frac{1}{n}\big) e^{i\theta - c|\theta|^{\alpha_{\beta}}}$ for a sufficiently small parameter $c > 0$. This yields
		$$
		\norm{T^n(I-T)^k} \le C \int_{-\pi}^{\pi} |\Gamma'_n(\theta)| |1-\Gamma_n(\theta)|^k \norm{R(\Gamma_n(\theta),T)} \, d\theta =: I_1 + I_2,
		$$
		where the integrals $I_1$ and $I_2$ are partitioned at $\theta_n := \big(\frac{c}{n}\big)^{1/\alpha_{\beta}}$, analogously to \eqref{defI1I2}. Since $T$ is $\alpha_{\beta}$-Ritt, the same arguments used in the proof of Theorem \ref{Nevanlinna} show that the integral $I_2$ satisfies
		$$
		I_2 \le C_2 n^{-\frac{1+k-\alpha_{\beta}}{\alpha_{\beta}}} = C_2 n^{\frac{\alpha_{\beta}-k-1}{\alpha_{\beta}}} = C_2 n^{\frac{\alpha+(\beta-1)(k+1)}{\alpha}},
		$$
		where $C_2$ depends on $\alpha, \beta,c$ and $k$.
		To estimate the integral $I_1$, we invoke the $(\alpha,\beta)$-RK condition. For $0 \le |\theta| \le \theta_n$, the contour satisfies $|\Gamma_n(\theta)| > 1$ and $|\Gamma_n(\theta)| - 1 \ge \frac{1-2c^2}{n}$. Consequently,
		\begin{align*}
			I_1 &\le C \int_{0 \le |\theta| \le \theta_n} e^{-nc|\theta|^{\alpha_{\beta}}} \frac{|1-\Gamma_n(\theta)|^{k-\alpha}}{(|\Gamma_n(\theta)|-1)^{\beta}} \, d\theta \\
			&\le C' n^{\beta} \int_{0 \le |\theta| \le \theta_n} e^{-nc|\theta|^{\alpha_{\beta}}} |1-\Gamma_n(\theta)|^{k-\alpha} \, d\theta,
		\end{align*}
		where $C'$ depends on $\alpha, \beta, c$, and $k$. 
		
		If $k - \alpha \ge 0$, we apply the estimate \eqref{estGamma_nmin1} and Lemma \ref{lem_gamma_estimates2} to obtain
		$$
		I_1 \le C' n^{\beta} \int_{0 \le |\theta| \le \theta_n} e^{-nc|\theta|^{\alpha_{\beta}}} \Big(|\theta|+\frac{1}{n}\Big)^{k-\alpha} \, d\theta \le C_1 n^{\beta - \frac{k+1-\alpha}{\alpha_{\beta}}}.
		$$
		If $-1 < k - \alpha < 0$, we utilize the lower bounds \eqref{ine12eitheta} along with \eqref{esti3fungamma} to find
		$$
		I_1 \le C' n^{\beta} \int_{0 \le |\theta| \le \theta_n} e^{-nc|\theta|^{\alpha_{\beta}}} |\theta|^{k-\alpha} \, d\theta \le C_1 n^{\beta - \frac{k+1-\alpha}{\alpha_{\beta}}},
		$$
		where $C_1$ depends on $\alpha, \beta, c$, and $k$.
		Observing that
		\begin{align*}
			\beta - \frac{k+1-\alpha}{\alpha_{\beta}} = \frac{\alpha + (\beta - 1)(k+1)}{\alpha},
		\end{align*}
		we conclude that $I_1 \le C_1 n^{\frac{\alpha + (\beta - 1)(k+1)}{\alpha}}$, completing the proof of the third item.
	\end{proof}

	\begin{remark}\label{rkcompar}
		\begin{enumerate}
			\item Setting $\beta = 0$ in Theorem \ref{abRK} provides the rates for an $\alpha$-Ritt operator with $\alpha > 1$:
			$$
			\norm{T^n(I-T)^k} = \begin{cases}
				O(n^{\alpha -k-1}), & \text{if } k < \alpha - 1, \\
				O(\log(n+1)), & \text{if } k = \alpha - 1, \\
				O(n^{\frac{\alpha - k-1}{\alpha}}), & \text{if } k > \alpha - 1.
			\end{cases}
			$$
			In particular, for $1 < \alpha < 2$, an $\alpha$-Ritt operator satisfies
			$$
			\norm{T^n(I-T)^k} =  \begin{cases}
				O(n^{\alpha-1}), & \text{if } k=0, \\
				O(n^{\frac{\alpha - k-1}{\alpha}}), & \text{if } k \in \mathbb{N}.
			\end{cases}
			$$
			This clarifies the link with Theorem \ref{Nevanlinna}: in this case, adding the Kreiss boundedness assumption only improves the rate at the boundary $k=0$, where it reduces the power growth from $O(n^{\alpha-1})$ down to $\norm{T^n} = O(n^{\frac{\alpha-1}{\alpha}})$.
			
			\item In the regime where $0 < \alpha, \beta < 1$ and $\alpha + \beta > 1$, Theorem \ref{abRK} yields for all $k \in \mathbb{N} \cup \{0\}$:
			$$
			\norm{T^n(I-T)^k} = O\big(n^{\frac{\alpha + (\beta-1)(k+1)}{\alpha}}\big).
			$$
			In particular, for $k=0$, our bound reads
			\begin{equation}\label{abcase0}
				\norm{T^n} = O(n^{\frac{\alpha+\beta -1}{\alpha}}).
			\end{equation}
			This improves the bounds available in the literature. For instance, in \cite[Corollary 3.3, Case 4]{MahRue}, the authors obtained $\norm{T^n} = O\big(n^{\frac{\alpha + \beta -1}{\alpha}}\log(n+1)\big)$ for $\alpha \in \mathbb{N}^{-1}:= \big\{\frac{1}{m}, \; m\in \N \}$, and, for $\alpha \notin \N^{-1}$,
			$$ 
			\norm{T^n} = \begin{cases} 
				O\Big(n^{(\alpha + \beta -1)(1+\lfloor\frac{1}{\alpha}\rfloor)}\Big), & \text{if } \frac{1}{\alpha}- \lfloor\frac{1}{\alpha}\rfloor \ge \frac{\alpha+\beta-1}{\alpha}, \\
				O\Big(n^{\lfloor\frac{1}{\alpha}\rfloor(\beta-1)+1}\Big), & \text{if }  \frac{1}{\alpha}- \lfloor \frac{1}{\alpha}\rfloor < \frac{\alpha+\beta-1}{\alpha}.
			\end{cases}
			$$
			Our bound \eqref{abcase0} unifies and improves these estimates for all parameters.
		\end{enumerate}
	\end{remark}

	Let $\alpha > 1$. If $T$ is $\alpha$-Ritt and Kreiss bounded, then for $1 < |\lambda| \le 2$ and any $\gamma \in (0,1)$, we have
	$$
	\norm{R(\lambda, T)} = \norm{R(\lambda, T)}^{\gamma} \norm{R(\lambda, T)}^{1-\gamma} \le \frac{C}{(|\lambda|-1)^{\gamma}|\lambda-1|^{\alpha(1-\gamma)}}.
	$$
	It follows that $T$ is $(\alpha(1-\gamma), \gamma)$-RK bounded for any $\gamma \in (0,1)$. 
	Moreover, if $T$ is $(\alpha(1-\gamma),\gamma)$-RK bounded for some $\gamma \in (0,1)$, then Lemma \ref{RKres} implies that $T$ is $\alpha$-Ritt. Consequently, for all $\gamma \in (0,1)$, we have the following chain of implications:
	\begin{equation}\label{implication}
		\text{  $\alpha$-Ritt and Kreiss bounded} \implies \text{ $(\alpha(1-\gamma),\gamma)$-RK bounded} \implies  \text{ $\alpha$-Ritt}.
	\end{equation}    
	In view of Remark \ref{abcase0} and the discussion above, we can provide a weaker condition than being both $\alpha$-Ritt and Kreiss bounded that still ensures the rate $\norm{T^n(I-T)^k} = O\big(n^{\frac{\alpha-k-1}{\alpha}}\big)$.
	
	\begin{cor}\label{CorEstalphaRitt}
		Let $\alpha > 1$ and let $\gamma \in (0,1)$ satisfy $\gamma > \frac{\alpha -1}{\alpha}$. Assume that $T$ is $(\alpha(1-\gamma), \gamma)$-RK bounded. Then for each $k \in \mathbb{N} \cup \{0\}$, 
		$$
		\norm{T^{n}(I-T)^k} = O\big(n^{\frac{\alpha-k-1}{\alpha}}\big).
		$$
	\end{cor}

	\begin{remark}
		If $1 \le \alpha < 2$ and $0 < \beta < 1$, it follows from the Theorem \ref{abRK} (with $k=1$) that if $T$ is $(\alpha,\beta)$-RK, then
		$$
		\norm{T^n(I-T)} = O\Big(n^{\frac{\alpha+2\beta -2}{\alpha}}\Big). 
		$$
		In \cite[Theorem 3.4] {MahRue}, the best bounds they obtain for this case are:
		$$
		\norm{T^n(I-T)} = \left\{
		\begin{array}{ll}
			O\big(n^{2\beta-1}\log(n+1)\big) & \mbox{if } \alpha = 1  \\
			O(n^{\beta}) & \mbox{if } 1<\alpha<2 \text{ and }  3-2\alpha\le\beta <1\\
			O(n^{2\alpha+2\beta-3}) & \mbox{if } 1<\alpha<2 \text{ and } 0<\beta < 3-2\alpha\\ 
		\end{array}
		\right.
		$$
		In all cases, our bounds are better, as we have the following comparisons: For $1< \alpha<2$ and $0<\beta <1$  
		$$
		\frac{\alpha+2\beta-2}{\alpha} - \beta = \Big(\frac{2}{\alpha}-1 \Big)(\beta-1) < 0
		$$
		and 
		\begin{align*}
			\frac{\alpha+2\beta-2}{\alpha} - (2\alpha+2\beta-3) =\frac{-2(\alpha -1)(\alpha+\beta-1)}{\alpha}  < 0.
		\end{align*}
		Finally, in the case $\alpha = 1$, the power rate in $n$ is the same as in \cite{MahRue}, but our estimate removes the logarithmic term.
	\end{remark}

	To conclude this section, we illustrate our results for the specific case $\alpha = 2$. Recall that from \eqref{implication}, for any $\gamma \in (0,1)$, we have the hierarchy:
	$$
	T \text{ is $2$-Ritt and Kreiss bounded} \implies T \text{ is $(2(1-\gamma),\gamma)$-RK bounded} \implies T \text{ is $2$-Ritt}.
	$$
	According to Theorem \ref{Nevanlinna} and Theorem \ref{abRK}, the asymptotic behavior of $\norm{T^n(I-T)^k}$ can be categorized as follows:
	\begin{itemize}
		\item \textbf{Case 1:} If $T$ is $2$-Ritt and Kreiss bounded, or if $T$ is $(2(1-\gamma),\gamma)$-RK bounded for $\frac{1}{2} < \gamma < 1$, then
		$$
		\norm{T^n(I-T)^k} = O\big( n^{\frac{1-k}{2}} \big), \quad k \in \mathbb{N} \cup \{0\}.
		$$
		
		\item \textbf{Case 2:} If $T$ is $(1,1/2)$-RK, the rate at the boundary $k=0$ involves a logarithmic factor:
		$$
		\norm{T^n(I-T)^k} = 
		\begin{cases}
			O(n^{1/2}\log(n+1)), & \text{if } k = 0, \\
			O\big( n^{\frac{1-k}{2}} \big), & \text{if } k \in \mathbb{N}.
		\end{cases}
		$$
		
		\item \textbf{Case 3:} If $T$ is $(2(1-\gamma),\gamma)$-RK with $0 < \gamma < \frac{1}{2}$, then since $2(1-\gamma) > 1$ and $2(1-\gamma) + \gamma - 1 = 1-\gamma$, we find:
		$$
		\norm{T^n(I-T)^k} = 
		\begin{cases}
			O(n^{1-\gamma}), & \text{if } k = 0, \\
			O\big( n^{\frac{1-k}{2}} \big), & \text{if } k \in \mathbb{N}.
		\end{cases}
		$$
		
		\item \textbf{Case 4:} Finally, if $T$ is merely a $2$-Ritt operator, the estimates are less restrictive:
		$$
		\norm{T^n(I-T)^k} = 
		\begin{cases}
			O(n), & \text{if } k = 0, \\ 
			O(\log(n+1)), & \text{if } k = 1, \\
			O\big( n^{\frac{1-k}{2}} \big), & \text{if } k \in \mathbb{N} \setminus \{1\}.
		\end{cases}
		$$
	\end{itemize}

	
	\section{Focus on $\beta$-Kreiss bounded operator}\label{SecBetaKreiss}

	We recall that a $\beta$-Kreiss bounded operator is a $(0,\beta)$-RK operator; that is, there exists a constant $C>0$ such that
	\begin{equation}\label{betaKreiss}
		\norm{R(\lambda,T)} \le \frac{C}{(|\lambda|-1)^{\beta}}, \quad 1 < |\lambda| < 2.
	\end{equation}
	The smallest constant $C$ satisfying \eqref{betaKreiss} is denoted by $K_{\beta}(T)$. The case $\beta < 1$ is of less interest in this context as it implies $\norm{T^n} = O(e^{-\omega n})$ for some $\omega >0$. The case $\beta = 1$ corresponds to the classical Kreiss boundedness condition.
	
	By Theorem \ref{abRK}, we know that if $T$ is $\alpha$-Ritt for $\alpha >1$, then 
	\begin{equation*}
		\norm{T^n(I-T)^k} = O(n^{\frac{\alpha - k-1}{\alpha}}), \quad \text{for } k > \alpha -1.
	\end{equation*}
	When $T$ is further assumed to be Kreiss bounded ($\beta=1$), this bound also holds for integers $k \le \alpha - 1$. This leads us to investigate the growth of $\norm{T^n(I-T)^k}$ for $k < \alpha - 1$ when $T$ is both $\alpha$-Ritt and $\beta$-Kreiss bounded with $\beta > 1$. Since $\alpha$-Ritt boundedness implies $\alpha$-Kreiss boundedness, we naturally assume $1 < \beta < \alpha$.
	
	\begin{prop}\label{PropbetaKreissalphaRitt}
		Let $\alpha >1$ and $1 < \beta < \alpha$. Suppose $T$ is $\alpha$-Ritt and $\beta$-Kreiss bounded. Then for any $k\in \mathbb{N}$ such that $k < \alpha - 1$, 
		$$
		\norm{T^n(I-T)^k} = O\left(n^{\beta \frac{\alpha - k - 1}{\alpha}}\right).
		$$
	\end{prop}
	
	\begin{proof}
		Let $k < \alpha - 1$, so that $k-\alpha < -1$. Following the  method from Theorem \ref{Nevanlinna}, we write:
		\begin{align*}\label{defI1I2bis}
			\nonumber \int_{ -\pi}^{\pi} |\Gamma'_n(\theta)||1-\Gamma_n(\theta)|^k\norm{R(\Gamma_n(\theta),T)} d\theta
			& = \int_{|\theta| \le \tilde{\theta}_n}\dots d\theta + \int_{\tilde{\theta}_n < |\theta| \le \pi}\dots d\theta  \\
			& =: I_1 + I_2,
		\end{align*}
		where we set $\tilde{\theta}_n := c n^{-\frac{\beta}{\alpha}}$ with $c < \frac{1}{\sqrt{2}}$. We recall that for $0\le |\theta| \ge \theta_n = c n^{-\frac{1}{\alpha}}$ we have $|\Gamma_n(\theta)| > 1$ and $|\Gamma_n(\theta)|-1 \le  \frac{1-c^2}{n}$, so that, since $\tilde{\theta}_n \le \theta_n$, these estimations are also valid for $0\le |\theta| \le \tilde{\theta}_n$. Using the $\beta$-Kreiss condition:
		$$
		I_1 \le C n^{\beta} \int_{0}^{\tilde{\theta}_n} \left(\theta + \frac{1}{n}\right)^k d\theta \le C' n^{\beta} \tilde{\theta}_n^{k+1} = C_1 n^{\beta - \beta\frac{k+1}{\alpha}} = C_1 n^{\beta \frac{\alpha-k-1}{\alpha}}.
		$$
		where $C_1$ depends on $\alpha, K_{\beta}(T),c$ and $k$.
		For the second integral, using the $\alpha$-Ritt condition:
		$$
		I_2 \le C \int_{\tilde{\theta}_n}^{\pi} \theta^{k-\alpha} d\theta \le C_2 \tilde{\theta}_n^{k-\alpha+1} = C_2 \left( n^{-\frac{\beta}{\alpha}} \right)^{k-\alpha+1} = C_2 n^{\frac{\beta}{\alpha}(\alpha - k - 1)},
		$$
		where $C_2$ depends on $\alpha,c$ and $k$.
		Both terms yield the same rate, which concludes the proof.
	\end{proof}

	\subsection{Connection to generalized Cesàro boundedness}
	One of the central themes in the ergodic theory of $T$ concerns the convergence and the growth properties of various means of the discrete semigroup $\{T^n\}_{n \ge 0}$. For $\gamma > -1$, the Cesàro means of order $\gamma$ (or Cesàro-$\gamma$ means) are defined by
	\begin{equation}\label{defSgamma}
		S_n^{\gamma}(T) := \sum_{j=0}^n A_{n-j}^{\gamma-1} T^j,
	\end{equation}
	where the binomial coefficients $A_n^\gamma$ are given by the generating relation $(1-z)^{-(\gamma+1)} = \sum_{n=0}^\infty A_n^\gamma z^n$ for $|z|<1$. These coefficients satisfy $A_n^\gamma = \binom{n+\gamma}{n} = \frac{(\gamma+1)(\gamma+2)\dots(\gamma+n)}{n!}$ and behave asymptotically as $A_n^\gamma \sim \frac{n^\gamma}{\Gamma(\gamma+1)}$ for $n \to \infty$. 
	
	Traditionally, an operator $T$ is called Cesàro-$\gamma$ bounded if $\|S_n^\gamma(T)\| = O(n^\gamma)$. In the context of $\beta$-Kreiss bounded operators, it is natural to investigate a more general condition where the growth rate also depends on the parameter $\beta$. This leads us to the following definition, which generalizes the classical notion.

	\begin{df}
		Let $\gamma > 0$ and $\beta \ge 1$. An operator $T \in B(X)$ is said to be $C_\gamma^{(\beta)}$-bounded if there exists a constant $C > 0$ such that
		\begin{equation*}\label{Cgamma}
			\|S_n^{\gamma}(T)\| \le C n^{\gamma +\beta - 1}, \quad n \in \N.
		\end{equation*}
		We denote by $C_{\gamma}^{(\beta)}(T)$ the smallest constant $C$ satisfying \eqref{Cgamma}. When $\beta = 1$, we recover the classical Cesàro-$\gamma$ boundedness.
	\end{df}
	\begin{remark}
		In the following, $C_1^{(\beta)}$-boundedness will play a central role. It is worth noting the connection between $C_1^{(\beta)}$-boundedness and the classical Cesàro-$\beta$  boundedness. Indeed, using the property $S_n^{\gamma + \delta+1}(T) = \sum_{j=0}^n A_{n-j}^{\delta} S_j^\gamma(T)$, we can write:
		$$
		S_n^1(T) = \sum_{j=0}^{n} A_{n-j}^{-\beta} S_j^{\beta}(T).
		$$
		If we assume that $T$ is Cesàro-$\beta$ bounded (i.e., $\|S_j^{\beta}(T)\| \le C j^\beta$), it follows that:
		$$
		\|S_n^1(T)\| \le C \sum_{j=0}^{n} |A_{n-j}^{-\beta}| j^{\beta} \le C n^{\beta} \sum_{j=0}^{n} |A_{n-j}^{-\beta}| \le C n^{\beta} \sum_{j=0}^{\infty} |A_{j}^{-\beta}|.
		$$
		For $\beta > 1$,  $\sum_{j=0}^{\infty} |A_{j}^{-\beta}| < \infty$ according to \cite[$(1\cdot16)$ p.77]{Zygmund}. Consequently, any Cesàro-$\beta$ bounded operator is necessarily $C_1^{(\beta)}$-bounded.
	\end{remark}
	
	Our first result connects $C_\gamma^{(\beta)}$-boundedness to Kreiss boundedness.
	\begin{prop}\label{CesaroImpliesbetaKreiss}
		Let $\beta, \gamma \ge 1$ and let $T \in B(X)$ be $C_\gamma^{(\beta)}$-bounded. Then $T$ is $\gamma+\beta$-Kreiss bounded.
	\end{prop}
	\begin{proof}
		Let $f(\lambda) := \sum_{n=0}^{\infty} \Big(\sum_{j=0}^n A^{\gamma}_{n-j}T^j \Big) \lambda^{-n}$. Then 
		\begin{align*}
			f(\lambda) &= \sum_{j=0}^{\infty } \sum_{n=j}^{\infty} A^{\gamma-1}_{n-j}T^j\lambda^{-n} = \sum_{j=0}^{\infty } T^j \sum_{n=0}^{\infty} A^{\gamma-1}_{n}\lambda^{-n-j} = \lambda \Big(\sum_{j=0}^{\infty } \frac{T^j}{\lambda^j} \Big)\Big( \sum_{n=0}^{\infty} A^{\gamma-1}_{n}\lambda^{-n}\Big) \\
			&= \lambda^{\gamma+1} R(\lambda,T)(\lambda-1)^{-\gamma }.
		\end{align*}
		Furthermore, 
		\begin{align*}
			\frac{|\lambda|^{\gamma+1}}{|\lambda-1|^{\gamma}}\norm{R(\lambda,T)}& = \norm{f(\lambda)}  \le \sum_{n=0}^{\infty} \norm{S_n^{\gamma}(T)}|\lambda|^{-n} \le C_{\gamma}^{(\beta)}(T)\sum_{n=0}^{\infty}n^{\gamma+\beta -1}   |\lambda|^{-n} \\
			&\le C  \sum_{n=0}^{\infty}A_n^{\gamma+\beta -1} = C\frac{|\lambda|^{\gamma+\beta}}{(|\lambda|-1)^{\gamma + \beta}}, 
		\end{align*}
		with $C$ depending only $C_{\gamma}^{(\beta)}(T)$. Finally we obtain
		$$
		\norm{R(\lambda,T)} \le C\frac{|\lambda|^{\beta-1}|\lambda-1|^{\gamma }}{(|\lambda| -1)^{\gamma + \beta}}
		$$
		which gives that $T$ is $\gamma+\beta$-Kreiss bounded.
	\end{proof}

	\begin{remark}
		If $T \in B(X)$ is Cesàro-$\gamma$ bounded for some $\gamma > 0$, it is well known that $\|T^n\| = O(n^{\gamma})$. However, if $T$ is additionally assumed to be $\alpha$-Ritt, combining Propositions~\ref{PropbetaKreissalphaRitt} and~\ref{CesaroImpliesbetaKreiss} yields the estimate
		\[
		\|T^n\| = O\left(n^{(\gamma+1)\frac{\alpha-1}{\alpha}}\right).
		\]
		In particular, this ensures that $\|T^n\| = o(n^{\gamma})$ provided that $\alpha < \gamma + 1$.
	\end{remark}
	
	Conversely the next result establishes that $\beta$-Kreiss boundedness is a sufficient condition for $C_\gamma^{(\beta)}$-boundedness for any $\gamma > 1$.
	
	\begin{prop}\label{PropBetaKreissImpliesCgammabetabounded}
		Let $T$ be a $\beta$-Kreiss bounded operator with constant $K_\beta(T)$. Then for each $\gamma > 1$, $T$ is $C_\gamma^{(\beta)}$-bounded. Furthermore $C_{\gamma}^{(\beta)}(T) $ depends only on $\gamma$ and $K_\beta(T)$.
	\end{prop}
	
	\begin{proof}
		Let $0 < \rho < 1$. By the Cauchy inversion formula and the generating function identity 
		\begin{equation}\label{identitySnTn}
			\sum_{m=0}^{\infty} z^m S_m^{\gamma}(T) = (1-z)^{-\gamma} (I-zT)^{-1}, \quad |z|<1,
		\end{equation}
		the means $S_n^\gamma(T)$ can be represented as:
		\begin{align*}
			2\pi \rho^n S_n^{\gamma}(T) &= \int_{-\pi}^{\pi} (1-\rho e^{it})^{-\gamma} e^{int} \sum_{j=0}^{\infty} (\rho e^{it})^j T^j \, dt \\
			&= \frac{1}{\rho} \int_{-\pi}^{\pi} e^{i(n+1)t} (1-\rho e^{it})^{-\gamma} R\Big(\frac{e^{it}}{\rho}, T\Big) \, dt.
		\end{align*}
		The $\beta$-Kreiss boundedness of $T$ implies $\|R(\lambda, T)\| \le K_\beta(T) (|\lambda|-1)^{-\beta}$ for $|\lambda| > 1$. Setting $\lambda = e^{it}/\rho$, we have $|\lambda|-1 = \frac{1-\rho}{\rho}$, which yields:
		\begin{align*}
			\|S_n^{\gamma}(T)\| &\le \frac{K_\beta(T)}{2\pi \rho^{n+1} (\frac{1}{\rho}-1)^{\beta}} \int_{-\pi}^{\pi} \frac{1}{|\rho e^{it}-1|^{\gamma}} \, dt \\
			&= \frac{K_\beta(T)}{2\pi \rho^{n+1-\beta + \gamma}(1-\rho)^{\beta}} \int_{-\pi}^{\pi} \frac{1}{| \frac{e^{it}}{\rho}-1|^{\gamma}} \, dt.
		\end{align*}
		Using Lemma \eqref{LemmaIntRes}, for $\gamma > 1$ and $r = 1/\rho > 1$, we have 
		\begin{equation}\label{ineMahRue} 
			\int_{-\pi}^{\pi} |1-re^{it}|^{-\gamma} dt \le C_\gamma (r-1)^{1-\gamma}.
		\end{equation} We arrive at
		$$
		\|S_n^{\gamma}(T)\| \le \frac{\rho^{\beta-1}C_\gamma K_\beta(T)}{2\pi\rho^{n+1}(1-\rho)^{\beta -1 + \gamma}}.
		$$
		Since $\rho^{\beta-1} \le 1$ for $0 < \rho < 1$, taking $\rho = 1-\frac{1}{n}$, we obtain $\|S_n^{\gamma}(T)\| = O(n^{\gamma+\beta-1})$, which proves \eqref{Cgamma}.
	\end{proof}
	
	\begin{remark}
		\begin{enumerate}
			\item If $T$ is $\beta$-Kreiss bounded, the limit case $\gamma = 1$ leads to 
			$$\|S^1_n(T)\| = O \big(n^{\beta} \log(n+1)\big).$$
			We show in Section~\ref{mKreissnotC_1^m} that this behavior is optimal. Specifically, we provide an example of a $\beta$-Kreiss bounded operator for which 
			\[
			\|S^{1}_n(T)\| \ge C^{-1} n^{\beta} \log(n+1).
			\]
			\item If $T$ is $(\alpha, \beta)$-RK with $\alpha > 0$, the same  argument shows that for $\gamma \ge 1$:
			$$
			\|S_n^{\gamma}(T)\| \le \frac{C}{\rho^{n +1-\beta +\gamma}(\rho-1)^{\beta}} \int_{-\pi}^{\pi} \frac{1}{| \frac{e^{it}}{\rho}-1|^{\alpha+\gamma}} \, dt.
			$$
			In particular, for $\gamma = 1$, using \eqref{ineMahRue} and $\rho = 1-\frac{1}{n}$, we find $\|S_n^{1}(T)\| = O( n^{\beta + \alpha})$. Thus, any $(\alpha, \beta)$-RK operator is $C_1^{(\alpha+\beta)}$-bounded. This implies that $\alpha$-Ritt operators ($\alpha > 1$) are $C_1^{(\alpha)}$-bounded, while for $\alpha=1$ (Ritt case), $T$ is power-bounded and thus Cesàro bounded.
		\end{enumerate}    
	\end{remark}
	
	\begin{prop}
		For an operator $T \in B(X)$, the following are equivalent:
		\begin{enumerate}
			\item $T$ is $\beta$-Kreiss bounded.
			\item For each $\gamma > 1$, $\lambda T$ is $C_{\gamma}^{(\beta)}$-bounded uniformly for all $\lambda \in \mathbb{T}$, i.e.,
			\begin{equation}\label{unifCgamma}
				\sup_{\lambda \in \mathbb{T}} C_{\gamma}^{(\beta)}(\lambda T) < \infty.
			\end{equation}
			\item Condition \eqref{unifCgamma} holds for at least one $\gamma > 1$.
		\end{enumerate}
	\end{prop}
	
	\begin{proof}
		$(i) \implies (ii)$ : According to Proposition \ref{PropBetaKreissImpliesCgammabetabounded}, this follows from the fact that $K_\beta(\lambda T) = K_\beta(T)$ for all $\lambda \in \mathbb{T}$,  since the constant $C_{\gamma}^{(\beta)}(\lambda T)$ depends only on $\alpha$ and $K_{\beta}(\lambda T)$. 
		
		$(ii) \implies (iii)$ is trivial. 
		
		To prove $(iii) \implies (i)$, let $\lambda = r e^{i\theta}$ with $r > 1$ and $\theta \in [-\pi, \pi]$. Using the identity \eqref{identitySnTn}, we can write the resolvent as:
		\begin{align*}
			R(r e^{i\theta}, T) &= e^{-i\theta} R(r, e^{-i\theta}T) = e^{-i\theta} \sum_{j=0}^{\infty} \frac{(e^{-i\theta}T)^j}{r^{j+1}} \\
			&= \left( 1-\frac{1}{r}\right)^{\gamma} \frac{e^{-i \theta}}{r} \sum_{j=0}^{\infty} \frac{S_j^{\gamma}(e^{-i\theta}T)}{r^j}.
		\end{align*}
		By assumption $(iii)$, there exists $\gamma > 1$ such that $\lambda T$ is $C_\gamma^{(\beta)}$-bounded uniformly in $\mathbb{T}$. Let $M = \sup_{\lambda \in \mathbb{T}} C_{\gamma}^{(\beta)}(\lambda T)$. Then, for any $j \in \mathbb{N}$, we have $\|S_j^{\gamma}(e^{-i\theta}T)\| \le M(j+1)^{\gamma+\beta-1}$. 
		
		Recall that for $\delta \ge 0$ and $|z|<1$, we have the power series $\sum_{j=0}^{\infty} z^j A_{j}^{\delta} = (1-z)^{-(\delta+1)}$ and the comparison $(j+1)^{\delta} \le D_{\delta} A_j^{\delta}$. It follows that:
		\begin{align*} 
			\|R(\lambda, T)\| &\le M \frac{(r-1)^{\gamma}}{r^{\gamma+1}} \sum_{j=0}^{\infty} \frac{(j+1)^{\gamma+\beta - 1}}{r^j} \le C \frac{(r-1)^{\gamma}}{r^{\gamma+1}} \sum_{j=0}^{\infty} \frac{A_{j}^{\gamma+\beta - 1}}{r^j} \\
			&=C \frac{(r-1)^{\gamma}}{r^{\gamma+1}} \left(1-\frac{1}{r}\right)^{-(\gamma +\beta)} = MD_{\beta} \frac{(r-1)^{\gamma}}{r^{\gamma+1}} \frac{r^{\gamma+\beta}}{(r-1)^{\gamma+\beta}} = C \frac{r^{\beta-1}}{(r-1)^{\beta}},
		\end{align*}
		where $C$ depends on $\gamma, \beta$ and $M$. Thus we obtain:
		$$
		\|R(\lambda, T)\| \le \frac{K}{(|\lambda|-1)^{\beta}}, \quad 1<|\lambda| <2.
		$$
		This shows that $T$ is $\beta$-Kreiss bounded, concluding the proof.
	\end{proof}

	\subsection{Uniform $\beta$-Kreiss boundedness}
	
	A strengthened condition of Kreiss boundedness was introduced in \cite{MSZ}, known as \textit{uniform Kreiss boundedness}. It is well established that an operator $T$ is uniformly Kreiss bounded if and only if 
	$$
	\sup_{\lambda \in \mathbb{T}}\sup_{n\in \mathbb{N}} \norm{(n+1)^{-1} S^1_n(\lambda T)} < \infty.
	$$
	In this section, we extend this concept to the $\beta$-order case.
	
	\begin{df}
		An operator $T \in B(X)$ is said to be uniformly $\beta$-Kreiss bounded if there exists a constant $C > 0$ such that
		\begin{equation}\label{unifbKreiss}
			\norm{\sum_{j=0}^n \frac{T^j}{\lambda^{j+1}}} \le \frac{C}{(|\lambda|-1)^{\beta}}, \quad \text{for all } 1 < |\lambda| < 2 \text{ and } n \in \mathbb{N}.
		\end{equation}
		We denote by $K_U^{(\beta)}(T)$ the smallest constant $C$ satisfying \eqref{unifbKreiss}.
	\end{df}
	
	Following the approach in \cite[Theorem 3.1]{MSZ} and \cite[Theorem 2]{GroHui}, we provide a characterization of this property in terms of uniform Cesàro means.
	
	\begin{prop}\label{KequivC}
		Let $T \in B(X)$ and $\beta \ge 1$. Then $T$ is uniformly $\beta$-Kreiss bounded if and only if $T$ is uniformly $C_1^{(\beta)}$-bounded on the unit circle, i.e.,
		\begin{equation}\label{unifbcesaro}
			\sup_{\lambda \in \mathbb{T}}\sup_{n\in \mathbb{N}} \norm{(n+1)^{-\beta} S^1_n(\lambda T)} < \infty.
		\end{equation}
	\end{prop}
	
	\begin{proof}
		Assume first that $T$ is uniformly $\beta$-Kreiss bounded. Let $\lambda \in \mathbb{T}$ and $r > 1$. By applying Abel's summation formula, we have:
		$$
		S_n^1(\lambda T) = \sum_{j=0}^n (\lambda T)^j = \sum_{j=0}^n r^{j} \frac{(\lambda T)^j}{r^j} = r^{n} \sum_{j=0}^n \frac{(\lambda T)^j}{r^j} + \sum_{j=0}^{n-1} (r^j - r^{j+1}) \sum_{l=0}^j \frac{(\lambda T)^l}{r^l}.
		$$
		Using the assumption \eqref{unifbKreiss} (noting that $\norm{\sum \frac{(\lambda T)^j}{r^{j}}} \le K_U^{(\beta)}(T)r (r-1)^{-\beta}$), we obtain:
		\begin{align*}
			\norm{S_n^1(\lambda T)} &\le K_U^{(\beta)}(T) \left( \frac{r^{n+1}}{(r-1)^{\beta}} + \sum_{j=0}^{n-1} \frac{r^{j+1}(r-1)}{(r-1)^\beta} \right) \\
			&= \frac{K_U^{(\beta)}(T)}{(r - 1)^{\beta}} \left( r^{n+1} + (r-1) \frac{r(r^n-1)}{r-1} \right) \le \frac{2 K_U^{(\beta)}(T) r^{n+1}}{(r-1)^\beta}.
		\end{align*}
		Setting $r = 1 + \frac{1}{n+1}$, we find $\norm{S_n^1(\lambda T)} \le C (n+1)^{\beta}$, where $C$ is independent of $n$ and $\lambda$. Thus, \eqref{unifbcesaro} holds.
		
		Conversely, assume that $T$ satisfies \eqref{unifbcesaro} with constant $M$. Let $\lambda = r e^{i\theta}$ with $1 < r < 2$. Then:
		$$
		\sum_{j=0}^n \frac{T^j}{\lambda^{j+1}} = \frac{e^{-i(n+1)\theta}}{r^{n+1}} S^1_n(e^{-i\theta} T) + \sum_{j=0}^{n-1} \left( \frac{1}{r^{j+1}} - \frac{1}{r^{j+2}} \right) e^{-i(j+1)\theta} S^1_j(e^{-i\theta} T).
		$$
		Applying the bound $\|S_j^1(e^{-i\theta} T)\| \le M(j+1)^\beta$, we get:
		$$
		\norm{\sum_{j=0}^n \frac{T^j}{\lambda^{j+1}}} \le M \left( \frac{(n+1)^{\beta}}{r^{n+1}} + \frac{r-1}{r} \sum_{j=0}^{n-1} \frac{(j+1)^{\beta}}{r^{j+1}} \right). 
		$$
		The function $x \mapsto x^{\beta} r^{-x}$ attains its maximum at $x = \beta/\log(r)$. Using the inequality $\log(r) \ge (r-1)/r$, we have:
		$$
		\frac{(n+1)^{\beta}}{r^n} \le \left( \frac{\beta}{\log(r)e} \right)^{\beta} \le \left( \frac{\beta r}{e} \right)^{\beta} \frac{1}{(r-1)^{\beta}}.
		$$
		Furthermore, using the comparison with binomial coefficients $ (j+1)^{\beta} \le  D_{\beta}A_j^\beta$, we have:
		$$
		\sum_{j=0}^{n-1} \frac{(j+1)^{\beta}}{r^{j+1}} \le \frac{D_{\beta}}{r} \sum_{j=0}^{\infty} \frac{A_j^{\beta}}{r^j} = \frac{D_{\beta}}{r} \left( 1 - \frac{1}{r} \right)^{-(\beta+1)} = \frac{D_{\beta} r^{\beta}}{(r-1)^{\beta+1}}.
		$$
		Combining these estimates, we obtain:
		$$
		\norm{\sum_{j=0}^n \frac{T^j}{\lambda^{j+1}}} \le  \frac{C}{(|\lambda|-1)^{\beta}}, \quad 1<|\lambda|<2,
		$$
		the constant $C$ is independent of $n$ and $\lambda$. This concludes the proof.
	\end{proof}

	\subsection{The case of positive operators}
	
	In this section, we focus on the class of positive operators. A well-known result (see e.g., \cite[Proposition 5.13.]{CCEL}) states that for a positive operator $T \in B(X)$, Cesàro boundedness is equivalent to Kreiss boundedness. Outside the setting of positive operators, this equivalence fails: there exist operators that are Cesàro bounded but not Kreiss bounded, and vice versa. We now establish an analogous result within our $\beta$-framework.
	
	\begin{prop}\label{positiveKreiss}
		Let $X$ be a Banach lattice and let $T \in B(X)$ be a positive operator. Then the following are equivalent:
		\begin{enumerate}
			\item $T$ is $C_1^{(\beta)}$-bounded.
			\item $T$ is uniformly $\beta$-Kreiss bounded.
			\item $T$ is $\beta$-Kreiss bounded.
		\end{enumerate}
	\end{prop}
	
	\begin{proof}
		$(i) \implies (ii)$: Let $\lambda \in \mathbb{T}$ and $x \in X$. Since $T$ is positive, we have for any $n \in \mathbb{N}$:
		$$
		\left| \sum_{j=0}^n (\lambda T)^j x \right| \le \sum_{j=0}^n |\lambda|^j T^j |x| = \sum_{j=0}^n T^j |x|.
		$$
		Taking the norm and using the fact that $\| |x| \| = \|x\|$ in a Banach lattice, we obtain:
		$$
		\|S_n^1(\lambda T)x\| \le \|S_n^1(T)|x|\| \le C_1^{(\beta)}(T) (n+1)^\beta \|x\|.
		$$
		This shows that $\sup_{\lambda \in \mathbb{T}} \sup_{n \in \mathbb{N}} \|(n+1)^{-\beta} S_n^1(\lambda T)\| < \infty$. By Proposition \ref{KequivC}, $T$ is uniformly $\beta$-Kreiss bounded.
		
		$(ii) \implies (iii)$ is immediate from the definitions.
		
		$(iii) \implies (i)$: Since $T$ is positive, for any $x \in X_+$ and any $r > 1$, we have:
		$$
		S_n^1(T)x = \sum_{j=0}^{n} T^j x \le r^{n+1} \sum_{j=0}^{n} \frac{T^j}{r^{j+1}} x \le r^{n+1} \sum_{j=0}^{\infty} \frac{T^j}{r^{j+1}} x = r^{n+1} R(r, T)x.
		$$
		Taking the norm and applying the $\beta$-Kreiss boundedness of $T$:
		$$
		\|S_n^1(T)x\| \le r^{n+1} \|R(r, T)\| \|x\| \le K_\beta(T) \frac{r^{n+1}}{(r-1)^{\beta}} \|x\|.
		$$
		By choosing $r = 1 + \frac{1}{n+1}$, we obtain:
		$$
		\|S_n^1(T)\| \le K_\beta(T) \left(1 + \frac{1}{n+1}\right)^{n+1} (n+1)^\beta \le e K_\beta(T) (n+1)^\beta,
		$$
		which concludes the proof that $T$ is $C_1^{(\beta)}$-bounded.
	\end{proof}

	\subsection{Sharpness of the power growth for $\beta$-Kreiss bounded operators}
	
	Recall that if $T$ is a Kreiss bounded operator on a Hilbert space, then 
	$$
	\norm{T^n} = o(n).
	$$
	Since a $\beta$-Kreiss bounded operator $T$ satisfies $\norm{T^n} = O(n^{\beta})$, we are naturally led to the following question: 
	
	\begin{q}\label{qbKreiss}
		Does every $\beta$-Kreiss bounded operator on a Hilbert space satisfy $\norm{T^n} = o(n^{\beta})$? 
	\end{q}
	
	In \cite{BonMul}, the authors constructed, for each $0 < \epsilon < 1$, a uniformly Kreiss bounded operator satisfying $\norm{T^n} \ge Cn^{1-\epsilon}$ for all $n \in \mathbb{N}$. Our approach to Question \ref{qbKreiss} relies on this construction; in what follows, we present it while extending it to the $\beta$-Kreiss setting, thereby providing a negative answer to this question.
	
	Let $\ell^2_N$ denote the space $\mathbb{C}^{2N}$ endowed with the norm $\norm{(x(j))_{j=1}^{2N}}_2 = \left( \sum_{j=1}^{2N} |x(j)|^2 \right)^{1/2}$. For $j \in \{ 1, \dots, 2N \}$, let $e_j \in \mathbb{C}^{2N}$ be the standard basis vectors. For $r>0$, we define the weights:
	\begin{equation*}
		\omega_j = \begin{cases} j^r, & \text{if } j=1,\dots,N, \\
			\frac{N^{2r}}{(2N-j+1)^r}, & \text{if }  j=N+1,\dots,2N,
		\end{cases}
	\end{equation*}
	and consider the weighted shift $T_N$ on $\ell^2_N$ given by:
	\begin{equation*}
		T_N e_j = \begin{cases}  \frac{\omega_{j+1}}{\omega_j} e_{j+1}, & \text{if } j=1,\dots,2N-1, \\
			0, &\text{if } j = 2N.
		\end{cases}
	\end{equation*}
	Since $\norm{T_N^m} = \max_{1 \le j \le 2N -m} \{ \frac{\omega_{j+m}}{\omega_j} \}$, it follows that:
	\begin{equation*}
		\norm{T_N^{m}} = \begin{cases} (m+1)^{r}, & \text{if } m= 1, \dots, N, \\
			\frac{N^{2r}}{(2N-m)^{2r} }, & \text{if } m = N+1, \dots, 2N - 1.
		\end{cases}
	\end{equation*}
	
	For $n \in \mathbb{N}$, the norm of the Cesàro mean is given by:
	$$
	\|S^1_n(T_N)\| = \sup\left\{\, |\langle S^1_n(T_N)x,y\rangle| : x \in \ell^2_N, y \in \ell^2_N, \ \|x\|_2=\|y\|_2=1 \right\}.
	$$
	Let $x=\sum_{j=1}^{2N}\alpha_j e_j$ and $y=\sum_{j=1}^{2N}\beta_j e_j$ satisfy $\norm{x}_2=1$ and $\norm{y}_2=1$. We decompose \(x=x_1+x_2\) and $y=y_1+y_2$ where $x_1, y_1$ are supported on $\{1, \dots, N\}$ and $x_2, y_2$ on $\{N+1, \dots, 2N\}$. Since $\langle S^1_n(T_N)x_2,y_1\rangle = 0$, we have:
	$$
	|\langle S^1_n(T_N)x,y\rangle| \leq A_n+B_n+C_n, 
	$$
	where $A_n:=\bigl|\langle S^1_n(T_N)x_1,y_1\rangle\bigr|$, $B_n:=\bigl|\langle S^1_n(T_N)x_2,y_2\rangle\bigr|$, and $C_n:=\bigl|\langle S^1_n(T_N)x_1,y_2\rangle\bigr|$. Following \cite{BonMul}, we have:
	$$
	A_n \le \sum_{j=1}^N\sum_{j'= j}^{\min(N,j+n)} |\alpha_j||\beta_{j'}|\Big(\frac{j'}{j}\Big)^r, \quad B_n \le \sum_{j=1}^N\sum_{j'= \max(1,j-n)}^{j} |\alpha_{2N-j+1}||\beta_{2N-j'+1}| \left( \frac{j'}{j} \right)^r
	$$
	and
	\begin{equation*}
		C_n \leq \begin{cases}  
			N^{2r} \left(\sum_{j=1}^{N}j^{-2r}\right)^{1/2} \left(\sum_{l=1}^{N}l^{-2r}\right)^{1/2}, & \text{if } n+1 \ge \frac{N}{2},\\
			2\left(\sum_{j=N-n+1}^{N}|\alpha_j|\right) \left(\sum_{j'=N+1}^{N+n}|\beta_{j'}|\right), & \text{if } n+1 <\frac{N}{2}.
		\end{cases}
	\end{equation*}
	We observe that:
	\begin{equation}\label{Cnparticular}
		\left(\sum_{j=N-n+1}^{N}|\alpha_j|\right) \left(\sum_{j'=N+1}^{N+n}|\beta_{j'}|\right) \le (n+1)^{1/2}\norm{x_1}_2(n+1)^{1/2}\norm{y_2}_2 \le n+1.
	\end{equation}
	
	Let $V$ be the unilateral weighted backward shift on $\ell^2(\mathbb{N})$, defined by $Vf_1 := 0$ and $Vf_k := \big(\frac{k}{k-1}\big)^r f_{k-1}$ for $k > 1$. Let $z=\sum_{j=1}^{\infty} z_j f_j$ with $\norm{z}_2 = 1$. Then, as shown in \cite[Theorem 2.1]{BBMP}:
	\begin{align*}
		\sum_{j=0}^{n} \|V^j z\|_2^2 &= \sum_{j=0}^{n} \sum_{l=j+1}^{\infty} |z_{l}|^2 \left( \frac{l}{l-j} \right)^{2r} 
		\le \sum_{l=1}^{2n} |z_{l}|^2 {l}^{2r} \sum_{j=0}^{l-1} (l-j)^{-{2r}} + 2^{2r}(n+1).
	\end{align*}
	Since \begin{equation}\label{sumjpowpr}
		\sum_{j=0}^{l-1} (l-j)^{-{2r}} = \sum_{j=0}^{l-1} j^{-{2r}}\le \begin{cases}
			\frac{l^{1-2r}}{1-2r} & \text{if } 0<r<1/2, \\
			1+\frac{1}{2r -1} & \text{if } r>1/2,
		\end{cases}
	\end{equation} it follows that:
	\begin{equation*}
		\sum_{j=0}^{n} \|V^j z\|_2^2 \le \begin{cases}
			\big(\frac{1}{1-2r} + 2^{2r}\big)(n+1) & \text{if } 0<r<1/2, \\
			\big(\frac{1}{2r-1} + 1+ 2^{2r}\big)(n+1)^{2r} & \text{if } r>1/2.
		\end{cases}
	\end{equation*}
	By Jensen's inequality, we have:
	\begin{equation*}
		\sum_{j=0}^{n} \|V^j z\|_2  \le (n+1) \left(\frac{1}{n+1} \sum_{j=0}^{n} \|V^j z\|^2_2 \right)^{1/2} \le \begin{cases}
			C(n+1) & \text{if } 0 < r < 1/2, \\
			C(n+1)^{r + 1/2} & \text{if } r > 1/2,
		\end{cases}
	\end{equation*}
	where $C$ depends on $r$. Since $\norm{S^1_n(V)z}_2 \le \sum_{j=0}^{n} \|V^j z\|_2$, it follows that $V$ is Cesàro bounded if $0 < r < 1/2$ and $C_1^{(r + 1/2)}$-bounded if $r > 1/2$. In particular, if $r > 1/2$, a simple comparison shows that $2r > r + 1 - 1/2$, and therefore $V$ is $C_1^{(2r)}$-bounded.
	
	The adjoint $V^*$, acting on $\ell^2(\mathbb{N})$, is given by the unilateral weighted forward shift $V^* f_j = \left( \frac{j+1}{j} \right)^{r} f_{j+1}$. For $u = \sum \delta_j f_j \in \ell^2(\mathbb{N})$ and $v = \sum \gamma_j f_j \in \ell^q(\mathbb{N})$ with $\|u\|_2 = \|v\|_2 = 1$ and non-negative coefficients, there exists a constant $C$ such that:
	\begin{equation}\label{estiV*}
		\sum_{j=1}^{\infty} \sum_{j' = j}^{\min(N, j+n)} \delta_j \gamma_{j'} \left( \frac{j'}{j} \right)^{r} = |\langle S^1_n(V^*)u, v \rangle| \le \|S^1_n(V)\| \le \begin{cases}
			C(n+1) &\text{if } 0 < r < 1/2, \\
			C(n+1)^{2r} &\text{if } r > 1/2.
		\end{cases}
	\end{equation}

	Applying \eqref{estiV*}, we immediately obtain:
	\begin{equation*}
		A_n \le \begin{cases} C(n+1), &\text{if } 0 < r < 1/2, \\
			C(n+1)^{2r}, &\text{if } r > 1/2,
		\end{cases}
		\quad \text{and} \quad B_n \le \begin{cases} C(n+1), &\text{if } 0 < r < 1/2, \\
			C(n+1)^{2r}, &\text{if } r > 1/2.
		\end{cases}
	\end{equation*}
	Furthermore, by applying \eqref{sumjpowpr} for $n+1 \ge \frac{N}{2}$ and \eqref{Cnparticular} for $n+1 < \frac{N}{2}$, we have:
	\begin{equation*}
		C_n \le \begin{cases} C(n+1), &\text{if } 0 < r < 1/2, \\
			C(n+1)^{2r}, &\text{if } r > 1/2.
		\end{cases}
	\end{equation*}
	This demonstrates that $T_N$ is Cesàro bounded if $0 < r < 1/2$ and $C_1^{(2r)}$-bounded if $r > 1/2$.

	Given $\beta > 0$, set $r = \beta/2$. On the space
	
	$$\mathcal{H} := \bigoplus_{N=1}^{\infty} \ell^2_N = \left\{ x = (x_N)_{N \ge 1} : x_N \in \ell^2_N \text{ and } \|x\|_{\mathcal{H}}^2 := \sum_{N=1}^{\infty} \|x_N\|_2^2 < \infty \right\},
	$$
	we define the block-diagonal operator $T$ by:
	\begin{equation*}
		T = \bigoplus_{N=1}^{\infty} T_N,
	\end{equation*}
	which acts on any $x = (x_N)_{N \in \mathbb{N}} \in \mathcal{H}$ as $T(x_1, x_2, \dots) = (T_1 x_1, T_2 x_2, \dots)$.

	The block-diagonal structure of $T$ implies that for any $n \in \mathbb{N}$, its $n$-th power and its first-order Cesàro mean are also block-diagonal, specifically:
	$$
	T^n = \bigoplus_{N=1}^{\infty} T_N^n \quad \text{and} \quad S_n^1(T) = \bigoplus_{N=1}^{\infty} S_n^1(T_N).
	$$
	Consequently, the norm of these operators can be computed as the supremum of the norms of their components:
	$$
	\|T^n\| = \sup_{N \ge 1} \|T_N^n\| \quad \text{and} \quad \|S_n^1(T)\| = \sup_{N \ge 1} \|S_n^1(T_N)\|.
	$$

	Then $T \in B(\mathcal{H})$ and:
	\begin{equation*}   
		\|S^1_n(T)\| = \sup_N \|S^1_n(T_N)\| \le \begin{cases} C(n+1), &\text{if } \beta <1, \\
			C(n+1)^{\beta}, &\text{if } \beta >1. 
		\end{cases}
	\end{equation*} 
	Now assume that $\beta >1$. Since $T$ is positive, Proposition \ref{positiveKreiss} implies $T$ is uniformly $\beta$-Kreiss bounded. Finally, for all $N \in \mathbb{N}$:
	$$
	\|T^{2N-1}\| \geq \|T_N^{2N-1}\| = N^{\beta} = \frac{1}{2^{\beta}}(2N)^{\beta}
	$$
	and
	$$
	\|T^{2N}\| \geq \|T_{N+1}^{2N}\| \geq \frac{\|T_{N+1}^{2N+1}\|}{\|T_{N+1}\|} > \frac{1}{2^{3\beta/2}}(2N+1)^{\beta}.
	$$
	This proves $\|T^N\| \geq C(N+1)^{\beta}$ for all $N \in \mathbb{N}$, answering Question \ref{qbKreiss} negatively:
	
	\begin{prop}
		For each $\beta > 1$, there exists a positive, uniformly $\beta$-Kreiss bounded operator $T_{\beta}$ on $\ell^2(\mathbb{N})$ such that $\norm{T_{\beta}^N} \ge CN^{\beta}$ for all $N \in \mathbb{N}$.
	\end{prop}

	\section{Examples}\label{SectionExemple}
	
	\subsection{Proof of Proposition \ref{PropNevExample}}\label{SectionExempleNev}

	Let $Y$ be the Banach space of analytic functions on $\mathbb{D}$ that are continuous on $\overline{\mathbb{D}}$, endowed with the norm
	$$
	\|f\|_Y = \|f\|_{\infty} + \|f'\|_1,
	$$
	where
	$$
	\|f\|_{\infty} = \sup_{|z| \le 1} |f(z)| \quad \text{and} \quad \|f'\|_1 = \int_{-\pi}^{\pi} |f'(e^{it})| dt.
	$$
	It turns out that $\|\cdot\|_Y$ is an algebra norm, meaning there exists a constant $C_Y > 0$ such that $\|fg\|_Y \le C_Y \|f\|_Y \|g\|_Y$.

	Let $T \in B(Y)$ be the multiplication operator defined by 
	$$
	(Tf)(z) = m(z)f(z) \quad \text{with} \quad m(z) := \frac{1+z}{2}.
	$$
	Then $T$ is Kreiss bounded and $2$-Ritt.

	Indeed let $|\lambda| \ge 1$ with $\lambda \neq 1$. For any $f \in Y$, the resolvent operator is given by
	$$
	R(\lambda, T)f(z) = (\lambda - m(z))^{-1}f(z).
	$$
	We have $\|R(\lambda, T)f\|_Y \le C_Y \|(\lambda - m)^{-1}\|_Y \|f\|_Y$. Thus, it suffices to estimate
	$$
	\|(\lambda - m)^{-1}\|_{\infty} \quad \text{and} \quad \|((\lambda - m)^{-1})'\|_1 = \frac{1}{2} \|(\lambda - m)^{-2}\|_1.
	$$
	First, using the residue formula, for $r > 1$ and $t \in [-\pi, \pi]$, or $r = 1$ and $t \neq 0$, we have:
	$$
	\int_{-\pi}^{\pi} \left| re^{it} - \frac{1+e^{i\theta}}{2} \right|^{-2} d\theta = \frac{2\pi}{r(r-\cos(t))}.
	$$
	Consequently, for $|\lambda| > 1$, we obtain:
	$$
	\|(\lambda - m)^{-2}\|_1 = \int_{-\pi}^{\pi} \left| \lambda - m(e^{i\theta}) \right|^{-2} d\theta \le \frac{C}{|\lambda|-1}.
	$$
	For $\lambda = e^{it}$ with $t \neq 0$, we have:
	$$
	\|(\lambda - m)^{-2}\|_1 = \int_{-\pi}^{\pi} \left| e^{it} - m(e^{i\theta}) \right|^{-2} d\theta = \frac{2\pi}{1-\cos(t)} \le \frac{C}{|\lambda - 1|^2}.
	$$
	Next, we estimate the $L^\infty$-norm. For $r > 1$ and $t, \theta \in [-\pi, \pi]$, we have:
	\begin{align*}
		\left| re^{it} - m(e^{i\theta}) \right|^{-1} &= \left| re^{it} - \frac{e^{i\theta}+1}{2} \right|^{-1} = \left| re^{i(t-\frac{\theta}{2})} - \cos\left( \frac{\theta}{2} \right) \right|^{-1} \\
		& \le \Big|r - cos\Big( \frac{\theta}{2}\Big) \Big|^{-1} \le \frac{1}{r-1}.
	\end{align*}
	For $t \in [-\pi, \pi] \setminus \{0\}$ and $\theta \in [-\pi, \pi]$, we have:
	$$
	|e^{it} - m(e^{i\theta})|^{-1} = 2 \left| (2e^{it}-1) - e^{i\theta} \right|^{-1} \le \frac{2}{\left| |2e^{it}-1| - 1 \right|} \le \frac{C}{|e^{it}-1|^2}.
	$$
	Therefore, for $|\lambda| > 1$, $\|(\lambda - m)^{-1}\|_{\infty} \le \frac{C}{|\lambda|-1}$, and for $\lambda \in \mathbb{T} \setminus \{1\}$, $\|(\lambda - m)^{-1}\|_{\infty} \le \frac{C}{|\lambda-1|^2}$.
	
	Finally, we obtain:
	$$
	\|R(\lambda, T)\|_{B(Y)} \le \frac{C}{|\lambda| - 1}, \quad |\lambda| > 1,
	$$
	which shows that $T$ is Kreiss bounded, and
	$$
	\|R(\lambda, T)\|_{B(Y)} \le \frac{C}{|\lambda - 1|^2}, \quad \lambda \in \mathbb{T} \setminus \{1\},
	$$
	which shows that $T$ is $2$-Ritt.

	Furthermore,  Moreover, it satisfies:
	\begin{equation}\label{estiexamnev} 
		\frac{1}{C} n^{\frac{1-k}{2}} \le \|T^n(I-T)^k\|_{B(Y)} \le C n^{\frac{1-k}{2}}, \quad k \ge 0
	\end{equation}
	where $C$  depends only on $k$. 
	
	Indeed by Theorem \ref{Nevanlinna}, we have the upper bound:
	$$
	\|T^n(I-T)^k\|_{B(Y)} \le C n^{\frac{1-k}{2}}, \quad k \in \mathbb{N},
	$$
	where $C$ depends only on $k$. To establish the lower bound, let $\mathbf{1}$ denote the constant function equal to $1$ on $\mathbb{D}$. Then,
	$$
	T^n(I-T)^k(\mathbf{1})(z) = m(z)^n(1-m(z))^k,
	$$
	so that
	$$
	\|T^n(I-T)^k(\mathbf{1})\|_Y = \|m^n(1-m)^k\|_{\infty} + \|(m^n(1-m)^k)'\|_1.
	$$
	First, we show that:
	\begin{equation}\label{mnorminfty}
		\|m^n(1-m)^k\|_{\infty} \ge \frac{1}{C} n^{-k/2},
	\end{equation}
	with $C$ depending only on $k$. We note that $|m(e^{i\theta})^n(1-m(e^{i\theta}))^k| = |\cos(\theta/2)|^n |\sin(\theta/2)|^k$ so that $\norm{m^n(1-m)^k}_{\infty} = \underset{\theta \in [0,\pi/2]}\sup{\cos^n(\theta)\sin^k(\theta)}$. Let $g(\theta) := \cos^n(\theta)\sin^k(\theta)$ for $\theta \in [0, \pi/2]$. The function $f(\theta) := \log(g(\theta)) = n\log(\cos\theta) + k\log(\sin\theta)$ satisfies:
	$$
	f'(\theta) = -n\tan(\theta) + k\cot(\theta), \quad \theta \in (0, \pi/2).
	$$
	The maximum of $f$ (and thus of $g$) is attained at $\theta_0 = \theta_0(n,k)$ satisfying $\tan^2(\theta_0) = k/n$. This implies $\cos^2(\theta_0) = \frac{n}{n+k}$ and $\sin^2(\theta_0) = \frac{k}{n+k}$. We then have:
	$$ 
	\|m^n(1-m)^k\|_{\infty} = g(\theta_0) = \left(\frac{n}{n+k}\right)^{n/2} \left(\frac{k}{n+k}\right)^{k/2} = \frac{n^{n/2}k^{k/2}}{(n+k)^{(n+k)/2}},
	$$
	which yields \eqref{mnorminfty}. Moreover, the derivative is given by:
	\begin{align*}
		(m^n(1-m)^k)'(z) &= nm'(z)m^{n-1}(z)(1-m(z))^k - km'(z)m^n(z)(1-m(z))^{k-1} \\
		&= m'(z)m^{n-1}(z)(1-m(z))^{k-1} \big( n(1-m(z)) - km(z) \big) \\
		&= \frac{n}{2} \left(\frac{1+z}{2}\right)^{n-1} \left(\frac{1-z}{2}\right)^{k-1} \left( 1 - \left(1+\frac{k}{n}\right)\frac{z+1}{2} \right),
	\end{align*}
	so that for $t \in [-\pi, \pi]$:
	$$
	(m^n(1-m)^k)'(e^{it}) = \frac{n}{2} \cos^{n-1}\left(\frac{t}{2}\right) \sin^{k-1}\left(\frac{t}{2}\right) \left( 1 - \left(1+\frac{k}{n}\right)\frac{e^{it}+1}{2} \right).
	$$
	Since
	\begin{align*}
		\left| 1 - \left(1+\frac{k}{n}\right)\frac{1+e^{it}}{2} \right| &\ge \left| \text{Im}\left( 1 - \left(1+\frac{k}{n}\right)\frac{1+e^{it}}{2} \right) \right| \\
		&= \left( 1+\frac{k}{n} \right) \frac{|\sin(t)|}{2} \ge \frac{|\sin(t)|}{2} =   \left| \cos\left(\frac{t}{2}\right) \right|\left| \sin\left(\frac{t}{2}\right) \right|  , 
	\end{align*}
	it follows that:
	\begin{align*}
		\|(m^n(1-m)^k)'\|_1 &\ge \frac{n}{2} \int_{0}^{\pi}  \left| \cos\left(\frac{t}{2}\right) \right|^{n} \left| \sin\left(\frac{t}{2}\right) \right|^{k} dt \\
		&\ge \frac{2^{k-1}n}{\pi^{k}} \int_{0}^{\pi} t^k \left| \cos\left(\frac{t}{2}\right) \right|^{n} dt \\
		&= \frac{2^{k-1}n}{\pi^{k}n^{(k+1)/2}} \int_{0}^{{\pi}\sqrt{n}} t^k \left| \cos\left(\frac{t}{2\sqrt{n}}\right) \right|^{n} dt.
	\end{align*}
	As $n \to \infty$, the integral converges to $\int_{0}^{\infty} t^k e^{-t^2/8} dt > 0$. Consequently, there exists a constant $C$ depending only on $k$ such that:
	$$
	\|(m^n(1-m)^k)'\|_1 \ge \frac{1}{C} n^{\frac{1-k}{2}}.
	$$
	Combined with \eqref{mnorminfty}, this establishes the lower bound in \eqref{estiexamnev}.

	\subsection{An $m$-Kreiss operator which is not $C_1^m$-bounded}\label{mKreissnotC_1^m}
	
	Let $\beta > 1$ and denote by $Y_m$ the space of analytic functions on $\mathbb{D}$ which are continuous on $\overline{\mathbb{D}}$ up to their $m$-th derivative, endowed with the norm:
	$$
	\norm{f}_{Y_m} := \sum_{j=0}^{m-1} \norm{f^{(j)}}_{\infty} + \norm{f^{(m)}}_1.
	$$
	Note that $\|\cdot\|_m$ is an algebra norm; thus, there exists a constant $C_m > 0$ such that for all $f, g \in Y_m$, $\|fg\|_m \le C_{Y_m} \|f\|_m \|g\|_m$. 
	The multiplication operator $M_z$ on $Y_m$, defined by $(M_z f)(z) = zf(z)$, satisfies:
	\begin{enumerate}[label = (\roman*)]
		\item $M_z$ is $m$-Kreiss bounded,
		\item $\|S^1_n(M_z)\|_{B(Y_m)} \ge C n^m \log(n+1)$.
	\end{enumerate}
	
	\begin{proof}
		$(i)$: Let $|\lambda| > 1$. The resolvent operator $R(\lambda, M_z)$ is the multiplication operator by the function $g_{\lambda}(z) = (\lambda - z)^{-1}$. We have:
		$$
		\|R(\lambda, M_z)f\|_m \le C_{Y_m} \|f\|_{Y_m} \|(\lambda - z)^{-1}\|_m.
		$$
		To show that $M_z$ is $m$-Kreiss bounded, we must estimate, for $|\lambda| > 1$:
		$$
		\|(\lambda-z)^{-1}\|_{\infty}, \dots, \|(\lambda-z)^{-m}\|_{\infty} \quad \text{and} \quad \|((\lambda-z)^{-1})^{(m)}\|_1.
		$$
		Note that the $m$-th derivative of $(\lambda-z)^{-1}$ is $m!(\lambda-z)^{-(m+1)}$. 
		Let $\lambda = re^{it}$ with $r > 1$. For any $j \in \{1, \dots, m\}$ and $\theta \in [-\pi, \pi]$, we have:
		$$
		|re^{it} - e^{i\theta}|^{-j} \le (r-1)^{-j} = \frac{1}{(|\lambda|-1)^j}.
		$$
		Furthermore, for the $L^1$ term, we have, by Lemma \eqref{LemmaIntRes}:
		$$
		\int_{-\pi}^{\pi} |re^{it} - e^{i\theta}|^{-(m+1)} d\theta \le \frac{C_m}{(r-1)^m}.
		$$
		Summing these estimates, we obtain $\|(\lambda - z)^{-1}\|_{Y_m} \le \frac{C}{(|\lambda|-1)^m}$, which proves that $M_z$ is $m$-Kreiss bounded.
		
		$(ii)$: To establish the lower bound, we test the operator on the constant function $\mathbf{1} \in Y_m$. We have:
		$$
		\|S_n^1(M_z)\|_{B(Y_m)} \ge \|S_n^1(M_z)(\mathbf{1})\|_m = \left\| \sum_{j=0}^n z^j \right\|_m \ge \left\| \left(\sum_{j=0}^n z^j\right)^{(m)} \right\|_1.
		$$
		Let $D_n(z) = \sum_{j=0}^n z^j$ denote the Dirichlet-type kernel. It is a well-known result in Fourier analysis that the $L^1$-norm of the $j$-th derivative of $D_n$ satisfies:
		$$
		C^{-1} n^j \log(n+1) \le \|D_n^{(j)}\|_1 \le C n^j \log(n+1)
		$$
		for every $j \in \mathbb{N} \cup \{0\}$. Taking $j=m$, we obtain:
		$$
		\|S_n^1(M_z)\|_{B(Y_m)} \ge C^{-1} n^m \log(n+1),
		$$
		which completes the proof.
	\end{proof}

	\subsection{On the sharpness of the $m$-Ritt growth estimates}
	
	We are interested in $m$-Ritt operators with $m \in \mathbb{N}$. It is known that in this case, 
	\begin{equation}\label{mRittesti}
		\norm{T^n(I-T)^k} = O(n^{m-k-1}), \quad k \in \{0, \dots, m-2\}.
	\end{equation}
	As noted in \cite{BorSpi2}, for any $m \in \mathbb{N}$, there exists an $m$-dimensional space $X$ and an operator $T \in B(X)$ that is $m$-Ritt such that 
	$$
	\norm{T^n} \le n^{\alpha -1}.
	$$
	
	Indeed, let $T_m \in \mathbb{C}^{m \times m}$ be defined by
	$$
	T_m =
	\begin{pmatrix}
		1 & 0 & 0 & \cdots & 0 \\
		1 & 1 & 0 & \cdots & 0 \\
		0 & 1 & 1 & \ddots & \vdots \\
		\vdots & \ddots & \ddots & \ddots & 0 \\
		0 & \cdots & 0 & 1 & 1
	\end{pmatrix},
	$$
	then according to \cite[Lemma]{BorSpi2}, its resolvent is given by
	$$
	(\lambda I-T_m)^{-1}=
	\begin{pmatrix}
		\frac{1}{\lambda-1} & 0 & 0 & \cdots & 0 \\
		\frac{1}{(\lambda-1)^2}& \frac{1}{\lambda-1} & 0 & \cdots & 0 \\
		\frac{1}{(\lambda-1)^3} & \frac{1}{(\lambda-1)^2} & \frac{1}{\lambda-1} & \ddots & \vdots \\
		\vdots & \vdots & \vdots & \ddots & 0 \\
		\frac{1}{(\lambda-1)^m} & \frac{1}{(\lambda-1)^{m-1}} & \frac{1}{(\lambda-1)^{m-2}} & \cdots & \frac{1}{\lambda-1}
	\end{pmatrix},
	$$
	which shows that $T_m$ is $m$-Ritt. Furthermore, writing for $m \ge 1$ and $n \ge 1$:
	$$
	J_{m,n} := \begin{pmatrix}
		1 & 0 & 0 & \cdots & 0 \\
		\binom{n}{1} & 1 & 0 & \cdots & 0 \\
		\binom{n}{2} & \binom{n}{1} & 1 & \ddots & \vdots \\
		\vdots & \vdots & \ddots & \ddots & 0 \\
		\binom{n}{m-1} & \binom{n}{m-2} & \cdots & \binom{n}{1} & 1
	\end{pmatrix} \in \mathbb{C}^{m \times m},
	$$
	where $\binom{n}{l} := 0$ if $l > n$, we clearly have $\norm{J_{m,n}} \le C n^{m-1}$. 
	Again, according to \cite[Lemma]{BorSpi2}, for $n \ge m$, we have
	$$
	T_m^n = J_{m,n},
	$$
	and then, for $k \le m-1$, 
	\begin{equation*}
		(T_m-I)^k T_m^n = \left(
		\begin{array}{cc}
			0 & 0 \\
			J_{m-k,n} & 0
		\end{array}
		\right),
	\end{equation*}
	which implies that $\norm{(I-T_m)^k T_m^n} \ge C n^{m-k-1}$.

	This demonstrates that the estimate \eqref{mRittesti} is sharp for $k \le m-2$.

	\subsection{A Cesàro-$m$ bounded operator which is not $m$-Kreiss}
	
	The Assani example $A = 
	\begin{pmatrix}
		-1 & 2 \\
		0 & -1 
	\end{pmatrix}$
	provides an example of a Cesàro bounded operator that is not Kreiss bounded. Following this idea, let 
	$$
	B_m =
	\begin{pmatrix}
		-1 & 0 & 0 & \cdots & 0 \\
		1 & -1 & 0 & \cdots & 0 \\
		0 & 1 & -1 & \ddots & \vdots \\
		\vdots & \ddots & \ddots & \ddots & 0 \\
		0 & \cdots & 0 & 1 & -1
	\end{pmatrix} \in \mathbb{C}^{(m+1) \times (m+1)}.
	$$
	Then, by \cite[Lemma 3.1 (b)]{BorSpi2}, the powers of $B_m$ are given by:
	$$
	B_m^n = 
	\begin{pmatrix}
		(-1)^{n} & 0 & \cdots & 0 \\
		(-1)^{n-1}\binom{n}{1} & (-1)^{n} & \cdots & 0 \\
		(-1)^{n-2}\binom{n}{2} & (-1)^{n-1}\binom{n}{1} & \ddots & \vdots \\
		\vdots & \vdots & \ddots & 0 \\
		(-1)^{n-m+1}\binom{n}{m} & (-1)^{n-m+2}\binom{n}{m-1} & \cdots & (-1)^{n}
	\end{pmatrix}.
	$$
	Now, since 
	$$
	\sum_{j=0}^n (-1)^{j}\binom{j}{m} = P_{m}(n),
	$$
	where $P_m$ is a polynomial of degree $m$, and for each $l \in \{ 1, \dots, m \}$,
	$$
	\sum_{j=0}^n (-1)^{j}\binom{j}{m-l} = Q_{m,l}(n),
	$$
	where $Q_{m,l}$ is a polynomial of degree at most $m-1$, it follows that there exists a constant $C > 0$ such that
	$$
	\frac{n^{m}}{C} \le \norm{S^1_n(B_m)} \le C n^{m}.
	$$ 
	Consequently, $B_m$ is Cesàro-$m$ bounded, which further implies that it is also $C_1^{(m)}$-bounded.
	
	Moreover, $B_m$ is not $m$-Kreiss bounded. Indeed, according to \cite[Theorem 4.2.]{BorSpi}, if $B_m$ were $m$-Kreiss bounded, it would satisfy:
	$$
	\norm{B_m^n} \le C_m n^{m-1},
	$$
	where $C_m$ depends only on $m$. However, the explicit form of $B_m^n$ shows that
	$$ 
	\norm{B_m^n} \ge D n^{m}
	$$
	for some $D > 0$. Therefore, $B_m$ cannot be $m$-Kreiss bounded.

	\subsection{Optimality of the growth bound for $(m,1)$-RK operators} 
	Let $m \in \mathbb{N}$. It is known that if $T$ is an $(m,1)$-RK operator, then 
	\begin{equation}\label{(m,1)-RK}
		\norm{T^{n}} = O(n^m \log(n+1)).
	\end{equation}
	We now provide an example showing that this estimate is sharp. 
	According to \cite{BorSpi2}, for any $N \in \mathbb{N}$, there exists a bounded operator $S_{N}$ on $X_N := (\mathbb{C}^{N^2}, \norm{\cdot}_N)$ which is $(m,1)$-RK with constant $R_{m,1}(S_N)$ depending only on $m$ (and thus independent of $N$), such that 
	$$
	\norm{S_N^N}_N \ge C N^{m} \log(N+1), \quad N \ge m+1,
	$$
	where $C$ depends only on $m$. The fact that $R_{m,1}(B_N)$ depends only on $m$ follows from  \cite[Lemma 3.1.]{BorSpi2}.
	Let $X$ be the $\ell^2$-direct sum of the spaces $X_N$:
	$$ 
	X := \left( \bigoplus_{N=m+1}^{\infty} X_N \right)_{\ell^2},
	$$
	endowed with the norm $\norm{x} = \left( \sum_{N=m+1}^{\infty} \norm{x_N}_N^2 \right)^{1/2}$. Let $B$ be the operator defined on $X$ by
	$$
	S := \bigoplus_{N=m+1}^{\infty} S_N.
	$$
	Then $S \in B(X)$ is $(m,1)$-RK since $\sup_N C_{RK}(B_N) < \infty$. Moreover, for each $N \ge m+1$, we have:
	$$
	\norm{S^N} \ge \norm{S_N^N}_N \ge C N^{m} \log(N+1).
	$$
	This demonstrates that the estimate \eqref{(m,1)-RK} is sharp.

	\subsection{On a generalization of the Tomilov-Zemanek example to higher order Cesàro means}
	
	Let $X$ be a Banach space. For $N \in \mathbb{N}$, we consider the space $\mathcal{X}_{N+1} = \bigoplus_{k=0}^N X$ equipped with the norm
	$$
	\|x_0 \oplus x_1 \oplus \dots \oplus x_N\|_{\mathcal{X}_{N+1}} := \sqrt{\sum_{j=0}^N \norm{x_j}^2}.
	$$
	Define the bounded linear operator $\mathcal{T}_{N+1}$ on $\mathcal{X}_{N+1}$ by the operator matrix:
	$$
	\mathcal{T}_{N+1} = \begin{pmatrix}
		T & 0 & 0 & \cdots & 0 \\
		I-T & T & 0 & \cdots & 0 \\
		0 & I-T & T & \cdots & 0 \\
		\vdots & \vdots & \ddots & \ddots & \vdots \\
		0 & 0 & \cdots & I-T & T
	\end{pmatrix}.
	$$
	Direct computations show that the $n$-th power of $\mathcal{T}_{N+1}$ is given by:
	$$
	\mathcal{T}_{N+1}^n = \begin{pmatrix}
		\rho_0(T) & 0 & 0 & \cdots & 0 \\
		\rho_1(T) & \rho_0(T) & 0 & \cdots & 0 \\
		\rho_2(T) & \rho_1(T) & \rho_0(T) & \cdots & 0 \\
		\vdots & \vdots & \vdots & \ddots & \vdots \\
		\rho_{N}(T) & \rho_{N-1}(T) & \rho_{N-2}(T) & \cdots & \rho_0(T)
	\end{pmatrix},
	$$
	where $\rho_i(T) := \binom{n}{i}T^{n-i}(I-T)^i$. Following the approach in \cite[Theorem 2.1]{TomZem}, we have the following result:
	
	\begin{prop}\label{PropNTomZem}
		For $N \ge 1$, let $\mathcal{T}_{N+1}$ be defined as above. Then:
		\begin{enumerate}
			\item $T$ is power bounded if and only if $\mathcal{T}_{N+1}$ is Cesàro-$N$ bounded.
			\item $(T^n )_{n\in \mathbb{N}}$ converges in the strong operator topology of $X$ if and only if the sequence $\big(n^{-N}S_n^{N}(\mathcal{T}_{N+1})\big)_{n\in \mathbb{N}}$ converges in the strong operator topology of of $\mathcal{X}_{N+1}$.
		\end{enumerate}   
	\end{prop}
	
	\begin{proof}
		We provide the proof for the case $N= 2$ (the $3 \times 3$ case). We have:
		$$
		S^2_n(\mathcal{T}_3) = \begin{pmatrix}
			S^2_n(\rho_0(T)) & 0 & 0 \\
			S^2_n(\rho_1(T)) & S^2_n(\rho_0(T)) & 0 \\
			S^2_n(\rho_2(T)) & S^2_n(\rho_1(T)) & S^2_n(\rho_0(T))
		\end{pmatrix}.
		$$
		Clearly, $S^2_n(\rho_0(T)) = S^2_n(T)$. For the first off-diagonal term, we compute:
		\begin{align*}
			S^2_n(\rho_1(T)) &= \sum_{j=0}^n (n-j+1)j T^{j-1}(I-T) = \sum_{j=0}^{n-1}(n-2j)T^j - nT^{n} \\
			&= 2 \sum_{j=0}^{n-1}(n-j+1)T^j - (n+2)\sum_{j=0}^{n-1} T^j - nT^n \\
			&= 2S_{n-1}^2(T) - (n+2)S_{n-1}^1(T) - nT^n. 
		\end{align*}
		For the second off-diagonal term, we have:
		\begin{align*}
			S^2_n(\rho_2(T)) &= \sum_{j=0}^n (n-j+1)\frac{j(j-1)}{2} T^{j-2}(I-T)^2 \\
			&= \sum_{j=0}^{n-1} (n - 3j - 1) T^j  + \frac{n(n-2)}{2}T^n \\
			&= 3\sum_{j=0}^{n-1}(n-j+1)T^j - (2n+4)\sum_{j= 0}^{n-1} T^j + \frac{n(n-2)}{2}T^n \\
			&= 3S_{n-1}^2(T) - (2n+4)S_{n-1}^1(T) + \frac{n(n-2)}{2}T^n.
		\end{align*}
		If $T$ is power-bounded, the above identities show that $n^{-2}S_n^2(\mathcal{T}_3)$ is uniformly bounded, so $\mathcal{T}_3$ is  Cesàro-$2$ bounded.
		
		Conversely, assume that $\mathcal{T}_3$ is  Cesàro-$2$ bounded. Let $x\in X$ with $\norm{x} = 1$. Since 
		$$
		S_n^2(\mathcal{T}_3)\begin{pmatrix}
			x \\
			0 \\
			0
		\end{pmatrix} = \begin{pmatrix}
			S_n^2(T)x \\
			S^2_n(\rho_1(T))x \\
			S^2_n(\rho_2(T))x
		\end{pmatrix},
		$$
		it follows that there exists a constant $C > 0$ such that
		$\norm{S_n^2(T)x} \le Cn^2$, $\norm{S^2_n(\rho_1(T))x} \le Cn^2$, and $\norm{S^2_n(\rho_2(T))x} \le Cn^2$. 
		
		The first estimate implies that $T$ is  Cesàro-$2$ bounded. This further implies that $\norm{S^1_n(T)x} \le C'n^2$ because $S_n^1(T) = S_n^2(T) - S_{n-1}^2(T)$. Using the previous expansion of $S_n^2(\rho_1(T))$, we have:
		$$ 
		\norm{n S_{n-1}^1(T)x} - \norm{2S_{n-1}^2(T)x - 2\sum_{j=0}^{n-1}T^jx} \le \norm{S_n^{2}(\rho_1(T))x} \le Cn^2,
		$$
		which leads to:
		$$
		\norm{nS_{n-1}^1(T)x} \le C'\left(n^2 + \norm{S_{n-1}^2(T)x} + \norm{S_{n-1}^1(T)x}\right) \le C''n^2. 
		$$ 
		Dividing by $n$, we obtain that $T$ is Cesàro bounded (i.e., $\norm{S_n^1(T)} = O(n)$). 
		
		In the same manner, using the fact that $T$ is Cesàro bounded and Cesàro-$2$ bounded, it follows that:
		$$
		\frac{n(n-2)}{2}\norm{T^nx} \le Cn^2 + 3\norm{S_{n-1}^2(T)x} + (2n+4)\norm{S_{n-1}^1(T)x} \le C'''n^2.
		$$
		This proves that $T$ is power bounded. Part $(2)$ regarding the strong convergence is obtained by following the same steps, replacing the norm estimates with strong operator topology convergence arguments.
	\end{proof}

	Following the ideas of \cite{TomZem}, we can now construct an operator $\mathcal{T}_{N+1}$ such that:
	\begin{enumerate}[label=(\roman*)]
		\item $\norm{\mathcal{T}^n_{N+1}} \ge Cn^{N}$,
		\item $\mathcal{T}_{N+1}$ is  Cesàro-$N$ bounded,
		\item $\big(n^{-N}S^{N}_n(\mathcal{T}_{N+1})\big)_{n\in \mathbb{N}}$ converges in the strong operator of $\mathcal{X}_N$,
		\item $\big(n^{-N}S^{N}_n(\mathcal{T}^*_{N+1})\big)_{n\in \mathbb{N}}$ does not converge in the in the strong operator of $\mathcal{X}^*_N$.
	\end{enumerate}
	
	Let $T$ be the forward shift on $\ell^2(\mathbb{N})$, which is power-bounded and converges strongly to $0$. Its adjoint $T^*$ (the backward shift) does not converge strongly. According to the proposition, $\mathcal{T}_{N+1}$ satisfies (ii) and (iii). Moreover, the growth condition (i) is satisfied since:
	$$\norm{\mathcal{T}_{N+1}^n} \ge \norm{\rho_{N}(T)} \ge \binom{n}{N}\sup_{\lambda \in \D} \big|\lambda^n(1-\lambda)^d\big| = 2^d \binom{n}{N} \ge Cn^{N}. $$
	
	Finally, for (iv), we observe that $\mathcal{T}^*_{N+1}$ is given by the upper triangular operator matrix:
	$$
	\mathcal{T}^*_{N+1} = \begin{pmatrix}
		T^* & I - T^* & 0 & \cdots & 0 \\
		0 & T^* & I - T^* & \cdots & 0 \\
		0 & 0 & T^* & \ddots & \vdots \\
		\vdots & \vdots & \ddots & \ddots & I - T^* \\
		0 & 0 & \cdots & 0 & T^*
	\end{pmatrix}.
	$$
	The same arguments used for $\mathcal{T}_{N+1}$ apply to this upper triangular structure. Consequently, the result of the Proposition \eqref{PropNTomZem} holds for $\mathcal{T}^*_{N+1}$ as well. Since $T^*$ does not converge strongly to $0$, it follows that the sequence $\big(n^{-N} S_n^N(\mathcal{T}^*_{N+1})\big)_{n \in \mathbb{N}}$ does not converge in the strong operator of $\mathcal{X}^*_N$.

	\section{Final remarks and open questions}

	The proves of \cite[Theorem 3.2. and 3.4]{MahRue} allows for the following slight improvement of Theorem \ref{Ritt}. 
	
	\begin{thm}\label{Ritt2}
		Let $T\in B(X)$ with $\sigma(T)\cap \mathbb{T}\subset \{1\}$. The following are equivalent:
		\begin{enumerate}[label=(\roman*)]
			\item $T$ is a Ritt operator;
			\item For some (or all) $m\in \mathbb{N}$ and $\alpha,\beta>0$ with $\alpha+\beta = m$, we have
			$\|R(\lambda,T)^m\| \le \frac{1}{|\lambda-1|^\alpha (|\lambda|-1)^\beta}$ for $1<|\lambda|<2$;
			\item $T$ is power-bounded and, for some (or all) $k\in \mathbb{N}$, $\|T^n(I-T)^k\| = O(n^{-k})$.
		\end{enumerate}
	\end{thm}
	
	The implication $(i) \Rightarrow (ii)$ is straightforward.
	
	For $(ii) \Rightarrow (iii)$: Recall that by the Riesz-Dunford calculus, for each $r > 1$ and $\varphi \in H^\infty(r\mathbb{D})$, we have
	\begin{equation}\label{eq:2.2}
		\varphi^{(m-1)}(T) = \frac{(m-1)!}{2\pi i} \int_{|\lambda|=\rho} \varphi(\lambda) R(\lambda, T)^{m} d\lambda, \quad 1 < \rho < r.
	\end{equation}
	Let $k \in \mathbb{N}$ such that $k\alpha > 2$. For $\varphi_1(z) = z^{n+km-1}$ and $\varphi_2(z) = z^{n+km} - \big(1+\frac{km}{n+1}\Big)z^{n+km-1}$, the derivatives are given by
	$$\varphi^{(km-1)}_1(z) = \frac{(n+km-1)!}{n!}z^n \quad \text{and} \quad \varphi^{(km-1)}_2(z) = \frac{(n+km)!}{(n+1)!}(z^{n+1} - z^n).$$
	Using \eqref{eq:2.2} and setting $\rho_1 = 1+\frac{1}{n}$ and $\rho_2 = 1+ \frac{km}{n+1}$, we obtain:
	$$
	T^n = \frac{1}{2\pi \binom{n+km-1}{km-1}} \int_{-\pi}^{\pi} \rho_1^{n+km} e^{i(n+km)t} R(\rho_1 e^{it}, T)^{km} dt,
	$$
	and
	$$
	T^n(I-T) = \frac{1}{2\pi \binom{n+km}{km-1}} \int_{-\pi}^{\pi} \rho_2^{n+km+1} e^{i(n+km)t} \left(\frac{1}{\rho_2}\Big(1+\frac{km}{n+1}\Big) - e^{it}\right) R(\rho_2 e^{it}, T)^{km} dt.
	$$
	From \cite[Lemma 3.1.]{MahRue}), we have:
	$$
	\int_{-\pi}^{\pi} \frac{1}{|(1+\frac{1}{n})e^{it}-1|^{k\alpha}} dt \le Cn^{k\alpha - 1}, \quad \int_{-\pi}^{\pi} \frac{|e^{it} -1|}{|(1+\frac{km}{n+1})e^{it}-1|^{k\alpha}} dt \le Cn^{k\alpha -2},
	$$
	and using the property $\binom{n+l}{l} \ge Cn^l$ for $l \in \mathbb{N}$ and $\alpha + \beta = m$, it follows that:
	\begin{align*}
		\norm{T^n} &\le C n^{-(km-1)} n^{k\beta} \int_{-\pi}^{\pi} \frac{1}{|(1+\frac{1}{n})e^{it}-1|^{k\alpha}}dt \\
		&\le C n^{-km+1} n^{k\beta} n^{k\alpha - 1} = Cn^{k(\alpha + \beta - m)} = C,
	\end{align*}
	and 
	\begin{align*}
		\norm{T^n(I-T)} &\le C n^{-(km-1)} n^{k\beta} \int_{-\pi}^{\pi} \frac{|e^{it}-1|}{|(1+\frac{km}{n+1})e^{it}-1|^{k\alpha}}dt \\
		&\le C n^{-km+1} n^{k\beta} n^{k\alpha - 2} = Cn^{k(\alpha + \beta - m)-1} = \frac{C}{n}.
	\end{align*}
	
	Regarding $(iii) \Rightarrow (i)$: If $T$ is power-bounded and $\norm{T^n(I-T)^k} = O(n^{-k})$, then the moment inequality (Lemma \ref{momentinequality}) implies $\norm{T^n(I-T)} = O(n^{-1})$. It is well known that a power-bounded operator satisfying this condition is a Ritt operator (see \cite[Theorem 4.5.4]{Nev2}).

	\begin{remark}
		An operator $T \in B(X)$ is said to be $\mathcal{R}$-Ritt if the set
		$$ \left\{ (\lambda - 1)R(\lambda, T) : |\lambda| > 1 \right\} $$
		is $\mathcal{R}$-bounded. We refer to \cite[Chapter 8]{HNVW2} for a detailed treatment of $\mathcal{R}$-boundedness. Using standard techniques from $\mathcal{R}$-boundedness and adapting classical Ritt theory, one can show that $T$ is $\mathcal{R}$-Ritt if and only if the set $\{T^n : n \ge 0\}$ is $\mathcal{R}$-bounded and the set $\{n T^n(I-T) : n \ge 1\}$ is $\mathcal{R}$-bounded. Consequently, the $\mathcal{R}$-bounded version of Theorem \ref{Ritt2} also holds. Further details on $\mathcal{R}$-Ritt operators can be found in \cite{LeMerdy}.
	\end{remark}
	
	\subsection*{Concluding Questions}

	Let $T\in B(X)$ be a Kreiss bounded operator, so that the moment inequality (Lemma \ref{momentinequality}) holds. Assume further that $T$ is $\alpha$-Ritt for some $\alpha > 1$. According to Nevanlinna's theorem, we have $\norm{T^n} = O(n^{\frac{\alpha-1}{\alpha}})$. 
	
	Let $f$ be a function bounded away from zero such that $\norm{T^n} \le f(n)$. By the moment inequality, for any $k \in \mathbb{N}$, we have:
	\begin{align*}
		\norm{T^n(I-T)} &\le C_k \norm{T^n}^{\frac{k-1}{k}} \norm{T^n(I-T)^k}^{\frac{1}{k}} \\
		&\le D_k f(n) n^{\frac{\alpha-k-1}{k\alpha}} = D_k f(n) n^{\frac{\alpha-1}{k\alpha}} n^{-\frac{1}{\alpha}}.
	\end{align*}
	This implies that for every $\varepsilon > 0$, $\norm{T^n(I-T)} = O(n^{-1/\alpha} f(n) n^{\varepsilon})$. This observation is particularly relevant for strongly Kreiss bounded operators on $L^p$, where $f(n) = \log^{\kappa}(n+1) n^{|1/2-1/p|}$. Note that if the constants satisfy $D_k \le D \cdot k$, then choosing $\log(n) \le k \le \log(n+1)$ leads to
	\begin{equation}\label{eq:refined_estimate}
		\norm{T^n(I-T)} = O(f(n) n^{-1/\alpha} \log(n+1)).
	\end{equation}
	This comparison with Seifert's theorem \cite[Corollary 3.1]{Sei} motivates the following question:
	
	\begin{q}
		Under the assumptions that $T$ is $\alpha$-Ritt for $\alpha >1$, Kreiss bounded and $\norm{T^n} \le f(n)$ for $f$ satisfying $\lim_{n \to \infty} f(n) = \infty$ with $f(n) = o\big(n^{\frac{\alpha-1}{\alpha}}\big)$, does the estimate \eqref{eq:refined_estimate} hold?
	\end{q}

	Another independent question concerns the sharpness of the estimates obtained in Theorem \ref{Nevanlinna}. As recalled in Section 6, Nevanlinna provided an example of a Kreiss bounded and $2$-Ritt operator showing that the estimates in Theorem \ref{Nevanlinna} are sharp. However, the case $\alpha \neq 2$ remains open:
	
	\begin{q}
		Let $\alpha > 1$ with $\alpha \neq 2$. Does there exist a Kreiss bounded and $\alpha$-Ritt operator such that for each $k \in \mathbb{N} \cup \{0\}$, there exists $C_k > 0$ satisfying for each $n\in \N$
		$$
		\norm{T^n(I-T)^k} \ge C_k n^{\frac{\alpha-k-1}{\alpha}} \,?
		$$
	\end{q}

	\noindent \textbf{Acknowledgments.}
	The author is grateful to Clément Coine for his insightful comments.


\begin{thebibliography}{99}
		
		
		\bibitem{ArnCun} L. Arnold and C. Cuny, \emph{On the growth rate of powers of a strongly Kreiss bounded operator on an $L^p$-space}, Studia Math. 288 (2026), no. 1, 1--25.
		
		\bibitem{BBMP} T. Berm\'udez, A. Bonilla, V. M\"uller and A. Peris, \emph{Ces\`aro bounded operators in Banach spaces}, J. Anal. Math. 140 (2020), no. 1, 187--206.
		
		\bibitem{BonMul} A. Bonilla and V. M\"uller, \emph{Kreiss bounded and uniformly Kreiss bounded operators}, Rev. Mat. Complut. 34 (2021), no. 2, 469--487.
		
		\bibitem{BorSpi} N. Borovykh and M. N. Spijker, \emph{Resolvent conditions and bounds on the powers of matrices, with relevance to numerical stability of initial value problems}, J. Comput. Appl. Math. 125 (2000), no. 1, 41--56.
		
		\bibitem{BorSpi2} N. Borovykh and M. N. Spijker, \emph{The sharpness of stability estimates corresponding to a general resolvent condition}, Linear Algebra Appl. 311 (2000), no. 1-3, 161--175.
		
		\bibitem{CCEL} G. Cohen, C. Cuny, T. Eisner and M. Lin, \emph{Resolvent conditions and growth of powers of operators}, J. Math. Anal. Appl. 487 (2020), no. 2, 124035, 24 pp.
		
		\bibitem{Cuny} C. Cuny, \emph{Resolvent conditions and growth of powers of operators on $L^p$ spaces}, Pure Appl. Funct. Anal. 5 (2020), no. 5, 1025--1038.
		
		
		
		\bibitem{DLV} C. Deng, E. Lorist and M. Veraar, \emph{Strongly Kreiss bounded operators in UMD Banach spaces}, Semigroup Forum 108 (2024), no. 3, 594–625.
		
		\bibitem{GroHui} J. J. Grobler and C. B. Huijsmans, \emph{Doubly Abel bounded operators with single spectrum}, Quaestiones Mathematicae 18 (1995), no. 4, 397--406.
		
		\bibitem{Haa} M. Haase, \emph{A functional calculus description of real interpolation spaces for sectorial operators}, Studia Math. 171 (2005), no. 2, 177--195.
		
		\bibitem{Haase2006} M. Haase, \emph{The Functional Calculus for Sectorial Operators}, Vol. 169, Operator Theory: Advances and Applications, Birkhäuser Verlag, Basel, 2006.
		
		\bibitem{HilPhi}
		E.~Hille and R.~S. Phillips, 
		\emph{Functional Analysis and Semi-groups}, 
		Rev. ed., American Mathematical Society Colloquium Publications, vol.~31, 
		American Mathematical Society, Providence, RI, 1957.
		
		\bibitem{HNVW2} T. Hytönen, J. van Neerven, M. Veraar and L. Weis, \emph{Analysis in Banach Spaces. Volume II: Probabilistic Methods and Operator Theory}, Ergebnisse der Mathematik und ihrer Grenzgebiete. 3. Folge, vol. 67, Springer Cham, 2017.
		
		
		\bibitem{LeMerdy} C. Le Merdy, \emph{$H^\infty$ functional calculus and square function estimates for Ritt operators}, Rev. Mat. Iberoam. 30 (2014), no. 4, 1149–1190.
		
		\bibitem{MahRue} A. Mahillo and S. Rueda, \emph{A Ritt--Kreiss condition: Spectral localization and norm estimates}, Studia Math. 276 (2024), 171--193.
		
		\bibitem{McCarthy} C. McCarthy, \emph{A strong resolvent condition does not imply power-boundedness}, Chalmers Institute of Technology and University of Gothenburg preprint 15, 1971.
		
		\bibitem{MSZ} A. Montes-Rodr\'iguez, J. S\'anchez-Alvarez and J. Zem\'anek, \emph{Uniform Abel--Kreiss boundedness and the extremal behaviour of the Volterra operator}, Proc. London Math. Soc. (3) 91 (2005), no. 3, 761--788.
		
		\bibitem{Nev2} O. Nevanlinna, \emph{Convergence of iterations for linear equations}, Lectures in Mathematics ETH Zürich, Birkhäuser Verlag, Basel, 1993.
		
		\bibitem{Nev} O. Nevanlinna, \emph{Resolvent conditions and powers of operators}, Studia Math. 145 (2001), no. 2, 113--134.
		
		\bibitem{Sei} D. Seifert, \emph{Rates of decay in the classical Katznelson--Tzafriri theorem}, J. Anal. Math. 130 (2016), no. 1, 329--354.
		
		\bibitem{TomZem} Y. Tomilov and J. Zemánek, \emph{A new way of constructing examples in operator ergodic theory}, Math. Proc. Camb. Phil. Soc. 137 (2004), 209--223.
		
		\bibitem{Zygmund} A. Zygmund, \emph{Trigonometric Series}, 2nd corrected edition (vol. I-II), Cambridge University Press, Cambridge, UK, 1968.
		
	\end{thebibliography}
\end{document}